\documentclass[11pt,letterpaper]{article}
\usepackage[T1]{fontenc}
\usepackage{amsmath}
\usepackage{amsfonts}
\usepackage{amssymb}
\usepackage{amsthm}
\usepackage{graphicx}
\graphicspath{{img/}}
\usepackage{booktabs}
\usepackage{hyperref}
\usepackage{breakcites}
\usepackage{multicol}
\usepackage{bbm}
\usepackage[ruled,vlined]{algorithm2e}
\usepackage{adjustbox}
\usepackage{multirow}
\usepackage{subfig}
\usepackage{float}
\usepackage{pgfplots}
\usepackage{todonotes}
\usepackage{enumitem}
\usepackage{authblk}
\usepackage{natbib}
\hypersetup{breaklinks=true}

\newtheorem{theorem}{Theorem}
\newtheorem{remark}{Remark}
\newtheorem{corollary}{Corollary}

\newtheorem{proposition}{Proposition}

\newtheorem{example}{Example}
\newtheorem{lemma}{Lemma}
\pgfplotsset{compat=1.18}
\usepgfplotslibrary{groupplots}
\usetikzlibrary{calc}

\usepackage{xcolor}

\newcommand{\p}{\mathbb{P}}
\newcommand{\E}{\mathbb{E}}
\newcommand{\R}{\mathbb{R}}

\newcommand{\id}[1]{\mathbbm{1}_{\{#1\}}}
\newcommand{\bm}[1]{\boldsymbol{#1}}

\tikzstyle{arrow} = [thick,->,>=stealth]
\usetikzlibrary{trees,shapes,decorations, arrows}
\usetikzlibrary{shapes,positioning,fit}

\usepackage[margin=1.00in]{geometry}

\renewcommand{\d}{\mathop{} \mathrm{d}}

\title{A Laplace-based perspective on conditional mean risk sharing}

\author{Christopher Blier-Wong}
\affil{Department of Statistical Sciences, University of Toronto, Canada, \href{mailto:christopher.blierwong@utoronto.ca}{christopher.blierwong@utoronto.ca}}

\date{\today}

\begin{document}
\maketitle

\begin{abstract}
The conditional mean risk-sharing (CMRS) rule is an important tool for distributing aggregate losses across individual risks, but its implementation in continuous multivariate models typically requires complicated multidimensional integrals. We develop a framework to compute CMRS allocations from the joint Laplace--Stieltjes transform of the risk vector. The LSTs of the allocation measures $\nu_i(B)=\E[X_i\id{S\in B}]$ are expressed as partial derivatives of the joint LST evaluated on the diagonal $t_1=\cdots=t_n$. When densities exist, this yields one-dimensional Laplace inversions for $f_S$ and $\xi_i$, and hence $h_i(s)=\xi_i(s)/f_S(s)$ on the absolutely continuous part, providing closed-form or semi-analytic solutions for a broad class of distributions. We also develop numerical inversion methods for cases where analytic inversion is unavailable. We introduce an exponential tilting procedure to stabilize numerical inversion in low-probability aggregate events. We provide several examples to illustrate the approach, including in some high-dimensional settings where existing approaches are infeasible.
\end{abstract}


\section{Introduction}\label{sec:intro}

Risk pooling is fundamental in insurance mathematics. Classical actuarial models focus on the aggregate loss $S = \sum_{i=1}^n X_i$, since solvency depends primarily on the total loss. Risk management practice increasingly demands allocation methods that assign losses back to individual participants: regulatory frameworks such as Solvency II require capital to be allocated to specific business lines, and decentralized peer-to-peer insurance systems require equitable risk-sharing rules.

The conditional mean risk sharing (CMRS) rule, proposed in \cite{denuit2012convex}, assigns to a nonnegative vector of losses $(X_1,\dots,X_n)$ with $S=\sum_{j=1}^n X_j$ the allocation
$$h_i(s)=\E[X_i\mid S=s], \qquad i=1,\dots,n.$$
The CMRS rule has a strong theoretical justification: \cite{jiao2022axiomatic} show that four axioms (actuarial fairness, risk fairness, risk anonymity, and operational anonymity) uniquely characterize CMRS among all risk-sharing rules, motivating its use in decentralized applications such as peer-to-peer insurance and cryptocurrency mining pools where anonymity and transparency are important constraints. CMRS also has a powerful economic property, in that it universally improves each participant's risk position in the convex order \cite{denuit2012convex}. By Jensen's inequality for conditional expectations, each participant's allocated loss is less variable than retaining their own risk, so every risk-averse agent prefers the CMRS allocation to retaining their own risk, regardless of their specific utility function. This willingness-to-join property is essential for the practical viability of any voluntary risk pool, since participants need not share a common preference, yet all benefit from pooling. Moreover, under CMRS, enlarging the pool is beneficial: individual contributions become increasingly concentrated around the pure premium as the number of participants grows, providing a natural bridge between collaborative risk sharing and classical insurance pricing \cite{denuit2021risk}.

In the continuous case, a direct approach expresses the conditional mean as the ratio
$$h_i(s) = \frac{\int_0^s x f_{X_i, S_{-i}}(x, s-x) \d x}{f_S(s)},$$
where $f_{X_i, S_{-i}}$ is the joint density of $(X_i, S_{-i})$. 
Although CMRS is conceptually straightforward, evaluating $s\mapsto \E[X_i\mid S=s]$ is numerically challenging \cite{denuit2022risksharing,boonen2025serial}. In high dimensions, this requires knowledge of the marginal density of $S$ and the joint density of $(X_i, S_{-i})$, where $S_{-i} = \sum_{j \neq i} X_j$. Unless the losses are independent or have convenient marginals, these densities are typically unknown (in special cases, such as mixed Erlang distributions under an FGM copula, the integral admits closed-form expressions \cite{barges2009tvarbased,cossette2012tvarbased}). Numerical integration is subject to the curse of dimensionality, while Monte Carlo simulation is inefficient for conditioning on the null event $\{S=s\}$ for continuous losses, and heavy-tailed distributions introduce additional numerical instability.

The size-biased transformation assigns weights to outcomes proportional to their values \cite{arratia2019size}. For a nonnegative random variable $X$ with $\E[X]\in(0,\infty)$, the size-biased distribution of $X$, denoted $X^*$, is defined by the measure
$$\d F_{X^*}(x) = \frac{x \,\d F_X(x)}{\E[X]}, \quad x \ge 0,$$
assigning greater probability to larger values. In \cite{denuit2019sizebiased}, the CMRS rule is expressed in terms of size-biased distributions:
$$h_i(s)=\E[X_i]\frac{f_{S_i^*}(s)}{f_S(s)}$$
where $S_i^* = X_i^* + \sum_{j \neq i} X_j$. This representation simplifies the computation of CMRS allocations by reducing it to a convolution: when densities exist, $f_{S_i^*}$ can be obtained by convolving the size-biased distribution of $X_i$ with $S_{-i}$. Some authors have used this representation to obtain explicit CMRS formulas under various distributional assumptions, investigate large-loss asymptotics, and construct polynomial approximations \citep{denuit2020largeloss,denuit2022conditional, denuit2022polynomial}.

A related line of work studies weighted distributions in which premiums and allocations are expressed as expectations under tilted laws. In \cite{furman2008weighteda}, weighted premiums are defined by replacing $F$ with $\d F_w(x)\propto w(x)\d F(x)$, which yields functionals of the form $\E[Xw(X)]/\E[w(X)]$ and recovers net, Esscher-type, and tail-based principles through different weight functions. In \cite{furman2008weighted}, the same perspective is extended to capital allocation through bivariate weighted functionals $\E[X w(Y)]/\E[w(Y)]$; \cite{furman2009weighted} then presents weighted distributions as a unifying mechanism for actuarial and economic pricing principles, and explores their relation to distortion-based constructions. In \cite{furman2021discussion}, the authors place size-biasing explicitly within this framework and emphasize univariate transform identities similar to the ones used in the present paper, though they do not consider the multivariate setting or the connection to CMRS.

Related computational methods have been developed for CMRS in the discrete setting. In \cite{blier-wong2025efficient}, the authors consider portfolios with lattice-supported components and study efficient evaluation of conditional means through generating functions: they work with the expected allocation $\E[X_i\id{S=s}]$ and relate ordinary generating functions to the multivariate probability generating function. The main advantage of the transform-based approach is that, on the transform scale, convolutions become products, which do not involve multidimensional integration and are therefore more amenable to numerical evaluation, especially in high dimensions. Some of the examples we consider in Section \ref{subsec:numerical-implementation} scale to dimensions as high as $n=100\;000$ participants, whereas the exact approach (which is available in closed form) scales only to $n=4$. The transform approach yields tractable algorithms, including recursive identities and FFT-based methods, for extracting the expected allocations. Their setup accommodates compound distributions, joint-transform dependence structures, and heavy-tailed marginals, with applications to peer-to-peer risk sharing and Euler-type capital allocation. The objective of this paper is to develop a transform-based framework for computing CMRS allocations in the continuous setting (although the main result applies to general nonnegative random variables, including discrete and mixed cases), where probability generating function methods no longer apply. We study the Laplace--Stieltjes transform of $(X_1,\ldots,X_n)$, which serves as the analogue of the probability generating function for mixed distributions, and obtain a transform-based representation of $\E[X_i\id{S=s}]$ which simplifies the computation of CMRS allocations. Building on this representation, we derive closed-form or semi-analytic CMRS allocations across several tractable model classes, develop numerical inversion methods, and show that exponential tilting improves stability for rare-event allocations.

Transform-based approaches have also been developed for related conditional expectations. For example, \cite{fonseca2026wisharta} considers loss models based on Wishart distributions and derives Fourier-based formulas for conditional tail functionals of the form $\E[X_i^q \mid S>s^*]$, expressed as one-dimensional integrals involving derivatives of moment-generating functions, reducing multivariate conditional quantities to low-dimensional transform computations through joint transform differentiation. 

The size-biased representation requires finite means, which is not necessary for CMRS itself: the allocation measures are well-defined without moment assumptions. It follows that, while our transform-based representation of CMRS is closely related to size-biased transforms, it admits an infinite mean, since the Laplace transform of the allocation measures is well-defined for each $t>0$ without requiring moments. That said, infinite-mean models are rarely practically relevant for risk sharing: under heavy-tailed Pareto losses, additional diversification can be welfare-reducing in first-order stochastic dominance \citep{chen2025diversification,chen2025infinitemean}, so we focus on the finite-mean setting in the applications.

The remainder of this paper is structured as follows. Section \ref{sec:main-result} states and proves the LST characterization of CMRS. In Section \ref{sec:closed-form}, we develop expressions for some closed-form cases, including for frailty dependence constructions, exponential dispersion family margins, and matrix-exponential distributions. Section \ref{sec:numerical} discusses numerical inversion methods. In Section \ref{sec:tilting}, we explain how to use exponential tilting to stabilize the computation of conditional means under rare events and provide numerical examples. Section \ref{sec:conclusion} concludes, and Appendix \ref{sec:proofs} contains the proofs of the main results.

\section{Transform-based CMRS representation}\label{sec:main-result}

\subsection{Conditional mean risk sharing and allocation measures}

Let $(\Omega,\mathcal F,\p)$ be a probability space and let $X_1,\dots,X_n:\Omega\to[0,\infty)$ be nonnegative random variables.
Set $S=\sum_{j=1}^n X_j$ and let $\mu_S$ denote its law. The CMRS rule assigns to component $i$ the allocation
\begin{equation}\label{eq:cmrs-def-intro}
h_i(s):=\E[X_i\mid S=s],\qquad s\ge 0.
\end{equation}
The vector $\bm h(s)$ satisfies the budget-balance condition $\sum_{i=1}^n h_i(s)=s$ for $\mu_S$-almost every $s$, also called the full allocation property in \cite{mcneil2015quantitative}. The CMRS rule coincides with the \textit{ex post} expected contribution of each participant at a fixed aggregate level \citep{denuit2012convex}.

In the discrete case, the conditional mean is well-defined at each point in the support of $S$, and the allocation can be expressed as the ratio $h_i(s)=\E[X_i\id{S=s}]/\p(S=s)$. In the continuous case, the conditional mean is defined as a Radon--Nikodym derivative, and it is preferable to work with the allocation measures
\begin{equation}\label{eq:alloc-measures-intro}
\nu_i(B):=\E[X_i\id{S\in B}],\qquad B\in\mathcal B([0,\infty)),
\end{equation}
and with their sum
\begin{equation}\label{eq:alloc-measures-S}
\nu(B):=\E[S\id{S\in B}]=\sum_{i=1}^n \nu_i(B),
\end{equation}
where $\mathcal B([0,\infty))$ is the Borel $\sigma$-algebra on $[0,\infty)$. The measure $\nu_i$ captures the expected contribution of component $i$ to the aggregate loss at different levels of $S$, and CMRS is then defined as $h_i(s)=\frac{\d\nu_i}{\d\mu_S}(s)$ for $\mu_S$-almost all $s$; we formalise this in Theorem \ref{thm:main-theorem}.

The joint Laplace--Stieltjes transform (LST) of $\bm X=(X_1,\dots,X_n)$ is
\begin{equation}\label{eq:joint-lst}
\mathcal L_{\bm X}(t_1,\dots,t_n):=\E\left[e^{-\sum_{j=1}^n t_j X_j}\right],
\qquad (t_1,\dots,t_n)\in[0,\infty)^n,
\end{equation}
and $\mathcal L_S(t):=\E[e^{-tS}]$. Since $0\le X_i\le S$, for each $s\ge 0$,
$$\nu_i([0,s])=\E\left[X_i\id{S\le s}\right]\le \E\left[S\id{S\le s}\right]\le s<\infty,$$
so $\nu_i$ is $\sigma$-finite. When $S$ admits a density $f_S$, Theorem \ref{thm:main-theorem} shows that $\nu_i$ is absolutely continuous with Lebesgue density $\xi_i$, that is $\nu_i(B)=\int_B \xi_i(s)\d s$. We use LST terminology at the measure level and reserve ``Laplace transform'' for the absolutely continuous density case.

\subsection{A Laplace representation of CMRS}

We now state the main result of this paper, which characterizes CMRS allocations in terms of the joint LST of the random vector $\bm X$ and the LSTs of the allocation measures $\nu_i$.

\begin{theorem}\label{thm:main-theorem}
Let $X_1,\dots,X_n$ be nonnegative random variables, $S=\sum_{j=1}^n X_j$, $\mu_S$ the law of $S$, and $\nu_i$, $\mathcal L_{\bm X}$ as defined in \eqref{eq:alloc-measures-intro}--\eqref{eq:joint-lst}. Then, the following hold: 
\begin{enumerate}[label=\textnormal{(\roman*)}, leftmargin=3.2em]
\item For any $(t_1,\dots,t_n)\in[0,\infty)^n$ with $t_i>0$, the partial derivative $\frac{\partial}{\partial t_i}\mathcal L_{\bm X}(t_1,\dots,t_n)$ exists as a two-sided derivative and satisfies
\begin{equation}\label{eq:partial-derivative}
\frac{\partial}{\partial t_i}\mathcal L_{\bm X}(t_1,\dots,t_n)
=-\E\left[X_ie^{-\sum_{j=1}^n t_j X_j}\right].
\end{equation}
If $t_i=0$, the right derivative exists (possibly $-\infty$). In particular, along the diagonal $t_1=\cdots=t_n=t$, define
\begin{equation*}
\mathcal L_i^*(t):=-\left.\frac{\partial}{\partial t_i}\mathcal L_{\bm X}(t_1,\dots,t_n)\right|_{t_1=\cdots=t_n=t},
\qquad t\ge 0,
\end{equation*}
so that $\mathcal L_i^*(0)=\E[X_i]$ (possibly $+\infty$) and
\begin{equation}\label{eq:Li-star-expectation}
\mathcal L_i^*(t)=\E\left[X_ie^{-tS}\right],\qquad t>0.
\end{equation}
\item Let $m_i:=\frac{\d\nu_i}{\d\mu_S}$ be the Radon--Nikodym derivative of $\nu_i$ with respect to $\mu_S$. Then $m_i(S)$ coincides almost surely with $\E[X_i\mid S]$. Also, $\E[X_i\mid S=s]=m_i(s)$ for $\mu_S$-almost every $s$.

\item If $S$ admits a density $f_S$ with respect to Lebesgue measure and $\xi_i$ denotes the (Lebesgue) density of $\nu_i$, then for almost every $s>0$ (with respect to Lebesgue measure),
\begin{equation}\label{eq:xi-factorization}
\xi_i(s)=\E[X_i\mid S=s]f_S(s),
\end{equation}
and the derivative $\mathcal L_i^*$ evaluated at $t_1=\cdots=t_n$ is the Laplace transform of $\xi_i$:
\begin{equation}\label{eq:Li-star-laplace}
\mathcal L_i^*(t)=\int_0^\infty e^{-ts}\xi_i(s)\d s,\qquad t>0.
\end{equation}
\end{enumerate}
\end{theorem}

In the independent case, the joint transform factorises and the partial derivatives at equal arguments admit convenient product formulas, as the following corollary shows.
\begin{corollary}\label{cor:independent}
If $X_1,\dots,X_n$ are independent, then for each $i$ and $t>0$,
\begin{equation}\label{eq:Li-star-independent}
\mathcal L_i^*(t)=\left(-\frac{\d}{\d t}\mathcal L_{X_i}(t)\right)\prod_{j\neq i}\mathcal L_{X_j}(t),
\qquad
\mathcal L_S(t)=\prod_{j=1}^n \mathcal L_{X_j}(t).
\end{equation}
\end{corollary}

When the marginal densities exist, these transform identities translate into convolution formulas in the original domain.
\begin{corollary}\label{cor:convolution-independent}
Assume $X_1,\dots,X_n$ are independent, and the densities $f_{X_i}$ and $f_{S_{-i}}$ exist, where $S_{-i}:=\sum_{j\ne i} X_j$. Then
$$f_S = f_{X_i}*f_{S_{-i}}, \qquad \xi_i = (x f_{X_i}(x))*f_{S_{-i}}.$$
\end{corollary}

\begin{remark}
Theorem \ref{thm:main-theorem} provides the transform-domain analogue of the size-biased identity in \cite{arratia2019size, denuit2019sizebiased}. Since multiplying a density by $x$ is equivalent to differentiating its Laplace transform, we have
$$\mathcal{L}_{X^*}(t) = \frac{-\frac{\d}{\d t}\mathcal{L}_X(t)}{\E[X]},$$
and the allocation transform $\mathcal{L}_i^*(t) = -\frac{\partial}{\partial t_i} \mathcal{L}_{\mathbf{X}}(t, \dots, t)$ is the unnormalized size-biased transform. 
\end{remark}

\subsection{Dealing with atoms in numerical inversion}\label{subsec:atoms}

The Radon--Nikodym formulation holds for mixed laws that may contain atoms (for example, Tweedie compound Poisson models which have a mass at $0$), but in the presence of atoms, the density-level representation of CMRS breaks down at the atom locations. In particular, for $s_0$ such that $\mu_S(\{s_0\})>0$, the CMRS at the atom is
$$\E[X_i\mid S=s_0]=\frac{\nu_i(\{s_0\})}{\mu_S(\{s_0\})}.$$
When $\mu_S$ or $\nu_i$ has atoms, the transforms $\mathcal{L}_S$ and $\mathcal{L}_i^*$ contain atomic contributions whose Laplace inversion yields Dirac masses rather than densities. The atomic part must therefore be extracted and handled separately before applying a density inversion rule.

Suppose $\mu_S$ has a finite set of atoms $\{s_j\}_{j\in J}$. One can write
$$\mathcal L_S(t)=\sum_{j\in J}\mu_S(\{s_j\})e^{-ts_j}+\int_{(0,\infty)}e^{-ts}f_S(s)\d s, \qquad \mathcal L_i^*(t)=\sum_{j\in J}\nu_i(\{s_j\})e^{-ts_j}+\int_{(0,\infty)}e^{-ts}\xi_i(s)\d s,$$
and define the atomic transform components
$$A_S(t):=\sum_{j\in J}\mu_S(\{s_j\})e^{-ts_j}, \qquad A_i(t):=\sum_{j\in J}\nu_i(\{s_j\})e^{-ts_j}.$$
One then applies numerical inversion only to the continuous remainders $t\mapsto \mathcal L_S(t)-A_S(t)$ and $t\mapsto \mathcal L_i^*(t)-A_i(t)$.
The CMRS ratio $h_i(s)=\xi_i(s)/f_S(s)$ is then computed for $s>0$ on the continuous part where $f_S(s)>0$, while at each atom one uses $h_i(s_j)=\nu_i(\{s_j\})/\mu_S(\{s_j\})$.

On the absolutely continuous part of $\mu_S$, CMRS is the ratio
$$\E[X_i\mid S=s]=\frac{\xi_i(s)}{f_S(s)},$$
where $f_S$ and $\xi_i$ are obtained by inverting the absolutely continuous remainders of $\mathcal L_S$ and $\mathcal L_i^*$ after subtracting atomic and singular terms. At atoms, one uses $\E[X_i\mid S=s_0]=\nu_i(\{s_0\})/\mu_S(\{s_0\})$ as the CMRS allocation.

For most insurance applications, the only atom is at $0$, corresponding to the event of no loss. In this case, one extracts the mass at $0$ from $\mathcal L_S$ and $\mathcal L_i^*$, applies numerical inversion to the continuous remainders, and then uses the mass at $0$ to compute the CMRS at $0$ (which is trivially zero since no loss implies no contribution from any component).

The point masses $\mu_S(\{s_j\})$ and $\nu_i(\{s_j\})$ entering $A_S$ and $A_i$ can be computed by discrete FFT or probability generating function methods \citep{blier-wong2025efficient}. 
In this framework, we assume that the locations of atoms are known a priori, which is typically the case in insurance applications, where the only atom is at $0$. This separation, therefore, allows FFT-based methods to handle the atomic part while LST inversion (Section~\ref{sec:numerical}) handles the continuous part.

\subsection{Allocation proportions and transform-level diagnostics}\label{subsec:diagnostics}

The following lemma gives a chain rule for the Radon--Nikodym derivatives of $\nu_i$ with respect to $\nu$ and $\mu_S$.
\begin{lemma}\label{lem:absolute-continuity}
For each $i\in\{1,\dots,n\}$, we have $\nu_i\ll \mu_S, \nu\ll \mu_S$, and $\nu_i\ll \nu.$ Moreover, for $\mu_S$-almost every $s>0$,
\begin{equation}\label{eq:rn-chain}
\frac{\d\nu_i}{\d\nu}(s)
=\frac{\E[X_i\mid S=s]}{s}.
\end{equation}
\end{lemma}

Lemma \ref{lem:absolute-continuity} allows one to work with the CMRS proportion
\begin{equation*}
\pi_i(s):=\begin{cases}
\E[X_i\mid S=s]/s,& s>0,\\
0,& s=0,
\end{cases}
\end{equation*}
which can be interpreted as the proportion of the aggregate loss allocated to component $i$ at level $s$. The function $\pi_i$ equals $\d\nu_i/\d\nu$ almost everywhere and satisfies $\sum_{i=1}^n \pi_i(s)=1$ for $\mu_S$-almost every $s>0$.
For numerical applications, proportions can be more stable than conditional means in low-probability events, where both the numerator and denominator are subject to numerical instability. 

The transform-level interpretation of CMRS leads to transform-level diagnostics that one may use as numerical consistency checks to verify that the numerical evaluation of the allocation measures $\nu_i$ and their transforms $\mathcal L_i^*$ is stable and accurate. We provide two diagnostic identities from the budget-balance condition and the diagonal identity. First, from $\sum_{i=1}^n \nu_i=\nu$, we have $\sum_{i=1}^n \frac{\d\nu_i}{\d\mu_S}(s)=s$
for $\mu_S$-almost every $s>0$. This is the measure-level budget-balance condition, which holds as an exact equality, and can be used to verify the numerical evaluation of the allocation measures $\nu_i$ and their derivatives. If $S$ has a density $f_S$ and each $\nu_i$ has a density $\xi_i$, then the balance condition gives
\begin{equation}\label{eq:budget-balance-density}
	\sum_{i=1}^n \xi_i(s)=sf_S(s)
\end{equation}
for Lebesgue-almost every $s>0$. In practice, \eqref{eq:budget-balance-density} serves as an effective diagnostic for numerical Laplace inversion. However, when $f_S$ and all $\xi_i$ are recovered with the same inversion routine and identical tuning parameters, any systematic bias common to all inversions may cancel in the sum, so \eqref{eq:budget-balance-density} is sensitive to discrepancies between the individual transforms $\{\mathcal L_i^*\}$ and $\mathcal L_S$ rather than to overall inversion accuracy. For this reason, we complement \eqref{eq:budget-balance-density} with normalization and moment checks (evaluated from the transform at $t=0$ when moments are finite), with sensitivity analysis of the recovered densities to the inversion parameters, and, when feasible, with comparisons across two inversion schemes. A second diagnostic follows from summing \eqref{eq:Li-star-expectation} over $i$: since $\sum_{i=1}^n \mathcal L_i^*(t)=\E[Se^{-tS}]$ for each $t>0$ and $\frac{\d}{\d t}\mathcal L_S(t)=-\E[Se^{-tS}]$ for $t>0$, we obtain
\begin{equation}\label{eq:diag-diagnostic}
\sum_{i=1}^n \mathcal L_i^*(t)=-\mathcal L_S'(t),\qquad t>0.
\end{equation}
which can also be used to verify the numerical evaluation of $\mathcal L_i^*$.

The analysis in this section yields a practical framework for computing CMRS allocations from the joint Laplace--Stieltjes transform, which we summarize in Algorithm \ref{alg:cmrs-procedure}. If the joint LST is available in closed form, the procedure yields closed-form or semi-analytic CMRS allocations. If the joint LST is not available in closed form but can be evaluated numerically, the procedure still applies and yields numerical CMRS allocations through numerical inversion of the transforms $\mathcal L_S$ and $\mathcal L_i^*$.

\begin{algorithm}[htbp]
\caption{Transform-based procedure to compute continuous CMRS allocations}\label{alg:cmrs-procedure}
\DontPrintSemicolon
\textbf{1.} Evaluate $\mathcal L_S(t)=\mathcal L_{\bm X}(t,\dots,t)$.\;
\textbf{2.} Evaluate $\mathcal L_i^*(t)=-\left.\frac{\partial}{\partial t_i}\mathcal L_{\bm X}(t_1,\dots,t_n)\right|_{t_1=\cdots=t_n=t}$.\;
\textbf{3.} If $S$ has atoms, subtract the corresponding atomic terms from $\mathcal L_S$ and $\mathcal L_i^*$.\;
\textbf{4.} Numerically invert the absolutely continuous remainders on $(0,\infty)$ to recover $f_S$ and $\xi_i$.\;
\textbf{5.} On the absolutely continuous part, define $h_i(s)=\xi_i(s)/f_S(s)$ when $f_S(s)>0$.\;
\textbf{6.} At each atom $s_j$ with $\mu_S(\{s_j\})>0$, compute $h_i(s_j)=\nu_i(\{s_j\})/\mu_S(\{s_j\})$.\;
\textbf{7.} Diagnostics:\;
\quad \textbf{(a)} On the continuous part, verify $\sum_i \xi_i(s)\approx s f_S(s)$.\;
\quad \textbf{(b)} At atoms, verify $\sum_i \nu_i(\{s_0\})=s_0 \mu_S(\{s_0\})$.\;
\quad \textbf{(c)} In the transform domain, verify $\sum_i \mathcal L_i^*(t)\approx -\mathcal L_S'(t)$.\;
\end{algorithm}

\section{Closed-form CMRS from the LST}\label{sec:closed-form}

In this section, we apply the transform-based representation of CMRS to derive closed-form or semi-analytic expressions for CMRS allocations in several tractable model classes, including frailty dependence constructions, exponential dispersion family margins, matrix-exponential distributions, and compound distributions. We will first discuss the frailty framework, which is closely related to Archimedean copulas and allows for convenient factorization of the joint LST, and then specialize to the case of exponential dispersion family margins under frailty, which yields compact LST expressions. Note that we consider dependent models in this section since the independence case is recovered as a special case of the frailty construction with degenerate distribution for $\Theta$.

\subsection{Frailty models and Archimedean dependence}\label{subsec:frailty-framework}

A broad class of dependence models 
is defined through a frailty construction, which specifies the risks as conditionally independent given a common factor $\Theta\ge 0$. This is also closely related to Archimedean copulas, where dependence is encoded by the LST of $\Theta$, see for example \cite{marshall1988families, oakes1989bivariate, mcneil2009multivariate}. This construction is useful when working with Laplace transforms, as in this paper, since the conditional independence given $\Theta$ allows for factorization of the joint LST into a product of conditional LSTs, and the unconditional LST is then obtained by integrating over the distribution of $\Theta$.

Assume that $\Theta$ is a non-negative random variable and that $(X_1,\dots,X_n)$ are conditionally independent given $\Theta$. For each $i$, let $\mathcal L_{X_i\mid\Theta=\theta}(t):=\E[e^{-tX_i}\mid \Theta=\theta]$ denote the conditional LST. Assume $(\theta,t)\mapsto \mathcal L_{X_i\mid\Theta=\theta}(t)$ is jointly measurable. Then, for all $(t_1,\dots,t_n)\in[0,\infty)^n$,
\begin{equation}\label{eq:frailty-joint-LST}
\mathcal L_{\bm X}(t_1,\dots,t_n)
= \E\left[\prod_{j=1}^n \mathcal L_{X_j\mid\Theta}(t_j)\right]
= \int_0^\infty \prod_{j=1}^n \mathcal L_{X_j\mid\Theta=\theta}(t_j) \mu_\Theta(\d\theta),
\end{equation}
where $\mu_\Theta$ is the law of $\Theta$.
In particular,
\begin{equation}\label{eq:frailty-sum-LST}
\mathcal L_S(t)=\int_0^\infty \prod_{j=1}^n \mathcal L_{X_j\mid\Theta=\theta}(t)\mu_\Theta(\d\theta),\qquad t\ge 0.
\end{equation}

Under the frailty model, $\mathcal L_i^*(t)=\E[X_i e^{-tS}]$ reduces to a one-dimensional integral over $\mu_\Theta$ involving derivatives of the conditional LSTs.

\begin{proposition}\label{prop:frailty-Lstar}
Assume $X_1,\dots,X_n$ are conditionally independent given $\Theta$ and $S=\sum_{j=1}^n X_j$. For $t>0$, we have
$$L_i^*(t)=\int \left(\mathbb{E}[X_i e^{-tX_i}\mid \Theta=\theta]\prod_{j\ne i}\mathbb{E}[e^{-tX_j}\mid \Theta=\theta]\right)\mu_\Theta(\d\theta).$$
\end{proposition}

When $\Theta$ acts as a random scale parameter, for example in the model $\bar F_{X_i\mid\Theta=\theta}(x)=\exp(-\theta x/\lambda_i)$, the unconditional survival is $\bar F_{X_i}(x)=\mathcal L_\Theta(x/\lambda_i)$ and the survival copula is Archimedean with generator $\phi(u)=\mathcal L_\Theta(u)$. 

\subsection{Exponential dispersion family and Tweedie margins under frailty}\label{subsec:edf-frailty}

Frailty constructions are especially convenient when conditional margins belong to the exponential dispersion family (EDF), which includes gamma, inverse Gaussian, and Tweedie distributions \cite{jorgensen1987exponential}, which are widely used in insurance applications. The conditional independence given $\Theta$ and the EDF structure allows for convenient factorization of the joint LST. Conditional on $\Theta=\theta$, assume that each random variable $X_i$ has a one-parameter EDF density of the form
	\begin{equation}\label{eq:edf-density}
	f_{X_i\mid\Theta=\theta}(x)
	=\exp\left\{
	\frac{x\eta_i(\theta)-\kappa_i(\eta_i(\theta))}{\phi_i}
	\right\}b_i(x,\phi_i),
			\qquad x\in\mathcal X_i\subseteq[0,\infty),
			\end{equation}
where $\phi_i>0$ is the dispersion parameter, $\eta_i(\theta)$ is the canonical parameter and $\kappa_i$ is the cumulant function. We assume $(b_i,\kappa_i)$ satisfy the standard EDF regularity conditions, see \cite{jorgensen1987exponential} for details. The mean and variance of $X_i$ conditional on $\Theta=\theta$ are given by $\E[X_i\mid\Theta=\theta]=\kappa_i'(\eta_i(\theta))$ and $\mathrm{Var}(X_i\mid\Theta=\theta)=\phi_i \kappa_i''(\eta_i(\theta))$, respectively.

Whenever $\eta_i(\theta)+\phi_i u$ lies in the domain of $\kappa_i$, the conditional moment-generating function of $X_i$ given $\Theta=\theta$ is 
$$M_{X_i\mid\Theta=\theta}(u):=\E[e^{uX_i}\mid\Theta=\theta] = \exp\left\{\frac{\kappa_i(\eta_i(\theta)+\phi_i u)-\kappa_i(\eta_i(\theta))}{\phi_i}\right\}.$$
Consequently, for all $t\ge 0$ such that $\eta_i(\theta)-\phi_i t$ lies in the domain of $\kappa_i$,
\begin{align}
\mathcal L_{X_i\mid\Theta=\theta}(t)
&=\exp\left\{\frac{\kappa_i(\eta_i(\theta)-\phi_i t)-\kappa_i(\eta_i(\theta))}{\phi_i}\right\},\label{eq:edf-LST}\\
\partial_t \mathcal L_{X_i\mid\Theta=\theta}(t)
&=-\kappa_i'(\eta_i(\theta)-\phi_i t)\mathcal L_{X_i\mid\Theta=\theta}(t).\label{eq:edf-derivative}
\end{align}
We obtain semi-analytic formulas for the aggregate and allocation transforms by integrating the conditional transforms over the mixing law of $\Theta$. 
\begin{proposition}\label{prop:edf-frailty-LST}
	Under the frailty and EDF settings of Proposition \ref{prop:frailty-Lstar}, for all $t>0$ such that $\eta_j(\theta)-\phi_j t$ lies in the domain of $\kappa_j$ for all $j$ and $\mu_\Theta$-almost every $\theta$,
\begin{align}
\mathcal L_S(t)
&=\int_0^\infty
\exp\left\{\sum_{j=1}^n\frac{\kappa_j(\eta_j(\theta)-\phi_j t)-\kappa_j(\eta_j(\theta))}{\phi_j}\right\}\mu_\Theta(\d\theta),\label{eq:edf-frailty-sum}\\
\mathcal L_i^*(t)
&=\int_0^\infty
\kappa_i'(\eta_i(\theta)-\phi_i t)
\exp\left\{\sum_{j=1}^n\frac{\kappa_j(\eta_j(\theta)-\phi_j t)-\kappa_j(\eta_j(\theta))}{\phi_j}\right\}\mu_\Theta(\d\theta).\label{eq:edf-frailty-Li}
\end{align}
\end{proposition}

\subsection{Mixed exponentials under frailty}\label{subsec:mex-arch}

We now consider the special case of mixed exponential margins: assume that $X_i\mid(\Theta=\theta) \sim \mathrm{Exp}(\theta/\lambda_i)$ for $i=1,\dots,n,$ independently for fixed scales $\lambda_i>0$ and for some positive random variable $\Theta$ with LST $\mathcal L_\Theta$. 

\begin{proposition}\label{thm:mex-arch}
Assume $\p(\Theta>0)=1$ and $\mathcal L_\Theta$ is differentiable on $(0,\infty)$.
Define, for distinct $\lambda_1,\dots,\lambda_n$, the coefficients
$$
A_k(\bm\lambda) := \prod_{m\neq k}\frac{\lambda_k}{\lambda_k-\lambda_m}.
$$
\begin{enumerate}[label=\textnormal{(\roman*)}, leftmargin=3.2em]
\item If $\lambda_1,\dots,\lambda_n$ are distinct, then for $s>0$ the aggregate random variable $S$ has density
$$
f_S(s)
= -\sum_{k=1}^n \frac{A_k(\bm\lambda)}{\lambda_k}
\mathcal L_\Theta'\left(\frac{s}{\lambda_k}\right).
$$
The corresponding unconditional expected allocation density is
\begin{align*}
\xi_i(s)
&= -\frac{s}{\lambda_i} A_i(\bm\lambda)
   \mathcal L_\Theta'\left(\frac{s}{\lambda_i}\right)
 + \sum_{k\neq i} A_k(\bm\lambda)\frac{\lambda_i}{\lambda_i-\lambda_k}
   \left( \mathcal L_\Theta\left(\frac{s}{\lambda_i}\right) - \mathcal L_\Theta\left(\frac{s}{\lambda_k}\right) \right).
	\end{align*}
	\item If $\lambda_1=\cdots=\lambda_n=\lambda$ and $S$ has a density, then for $\mu_S$-almost every $s>0$,
	$$
	\xi_i(s) = \frac{s}{n} f_S(s),
	\qquad
	\E[X_i\mid S=s] = \frac{s}{n}.
$$
\end{enumerate}
\end{proposition}

When the parameters $\lambda_i$ are all distinct, the conditional expressions for the density or expected allocations involve sums of exponential terms. If some $\lambda_i$ coincide, indices sharing the same $\lambda$ contribute an Erlang block in the sum, so the conditional law of the aggregate random variable is no longer a function of exponential terms but admits a mixed Erlang representation. The same applies to the expected allocations: conditionally on $\Theta$, expected allocations also admit an Erlang mixture representation, with the groups containing duplicated $\lambda$ parameters generating the higher-order Erlang components. These expressions are easy to derive by applying the derivative identity to the joint LST and evaluating the resulting integrals, but are omitted for brevity. 

\begin{example}   
   Assume that $\Theta$ is Gamma distributed with shape parameter $\alpha>0$ and rate parameter $1$, so that $\mathcal L_\Theta(u)=(1+u)^{-\alpha}$ for $u\ge 0$. One can show that, for each $i$, $X_i$ has a Pareto type II (Lomax) margin with shape $\alpha$ and scale $\lambda_i$. Since
\begin{equation}\label{eq:clayton-lst-derivative}
\mathcal L_\Theta'(u)=-\alpha(1+u)^{-\alpha-1},\qquad u>0,
\end{equation}
for distinct $\lambda_1,\dots,\lambda_n$,
$$
A_k(\bm\lambda):=\prod_{m\neq k}\frac{\lambda_k}{\lambda_k-\lambda_m}.
$$
Substituting \eqref{eq:clayton-lst-derivative} into Proposition \ref{thm:mex-arch}(i) gives, for $s>0$,
\begin{align*}
   f_S(s)&=\alpha\sum_{k=1}^n \frac{A_k(\bm\lambda)}{\lambda_k} \left(1+\frac{s}{\lambda_k}\right)^{-\alpha-1};\\
\xi_i(s) &=\alpha A_i(\bm\lambda)\frac{s}{\lambda_i}\left(1+\frac{s}{\lambda_i}\right)^{-\alpha-1} + \sum_{k\neq i} A_k(\bm\lambda)\frac{\lambda_i}{\lambda_i-\lambda_k} \left[\left(1+\frac{s}{\lambda_i}\right)^{-\alpha}-\left(1+\frac{s}{\lambda_k}\right)^{-\alpha}\right].
\end{align*}

\end{example}

\begin{example}

   Let $\Theta$ have a Lévy distribution (stable distribution with shape parameter $\frac12$) with LST
$$\mathcal L_\Theta(u):=\E[e^{-u\Theta}] = \exp\left(-\kappa\sqrt{u}\right),\qquad u\ge 0,$$
for some $\kappa>0$. One can show that, for each $i$, $X_i$ is Weibull distributed with shape $1/2$ and scale $\lambda_i/\kappa^2$, with density
$$f_{X_i}(x)=\frac{\kappa}{2\sqrt{\lambda_i x}}\exp\left(-\kappa\sqrt{\frac{x}{\lambda_i}}\right),\qquad x>0.$$

Since
\begin{equation}\label{eq:levy-lst-derivative}
\mathcal L_\Theta'(u) = -\frac{\kappa}{2\sqrt{u}}\exp\left(-\kappa\sqrt{u}\right),\qquad u>0,
\end{equation}
Substituting \eqref{eq:levy-lst-derivative} into Proposition \ref{thm:mex-arch}(i) yields
\begin{align*}
   f_S(s)&=\frac{\kappa}{2\sqrt{s}}\sum_{k=1}^n \frac{A_k(\bm\lambda)}{\sqrt{\lambda_k}}\exp\left(-\kappa\sqrt{\frac{s}{\lambda_k}}\right);\\
\xi_i(s) &=\frac{\kappa}{2}A_i(\bm\lambda)\sqrt{\frac{s}{\lambda_i}}
\exp\left(-\kappa\sqrt{\frac{s}{\lambda_i}}\right)  + \sum_{k\neq i} A_k(\bm\lambda)\frac{\lambda_i}{\lambda_i-\lambda_k}
\left(\exp\left(-\kappa\sqrt{\frac{s}{\lambda_i}}\right) - \exp\left(-\kappa\sqrt{\frac{s}{\lambda_k}}\right)\right).
\end{align*}
\end{example}

\subsection{Matrix-exponential margins}\label{subsec:me-example}

Matrix-exponential (ME) distributions are the nonnegative distributions with rational LST \cite{bladt2017matrixexponential}. They include phase-type (PH) distributions, hyperexponential mixtures, and mixtures of Erlang distributions, all of which play an important role in insurance applications; see \cite{willmot2011risk, cheung2022multivariate}. A nonnegative random variable $X$ is said to admit an ME representation of order $p$, written
$X\sim \mathrm{ME}_p(\bm{\alpha},\mathbf{T},\bm{u})$, if its density on $(0,\infty)$ can be written as
\begin{equation*}
f_X(x)=\bm{\alpha} e^{\mathbf{T}x} \bm{u},\qquad x>0,
\end{equation*}
for some row vector $\bm{\alpha}\in\R^{1\times p}$, column vector $\bm{u}\in\R^{p\times 1}$, and matrix
$\mathbf{T}\in\R^{p\times p}$ whose eigenvalues have strictly negative real parts. The Laplace--Stieltjes transform is then the rational function
\begin{equation}\label{eq:me-lst}
\mathcal L_X(t)=\E \left[e^{-tX}\right]
=\p(X=0)+\bm{\alpha} (t\mathbf{I}-\mathbf{T})^{-1}\bm{u},\qquad t\ge 0,
\end{equation}
Differentiating \eqref{eq:me-lst} with respect to $t$ yields
\begin{equation*}
-\frac{\d}{\d t}\mathcal L_X(t)=\E \left[Xe^{-tX}\right]
=\bm{\alpha} (t\mathbf{I}-\mathbf{T})^{-2}\bm{u},\qquad t\ge 0.
\end{equation*}
More generally, for $k\in\mathbb N$ one has
$\E[X^k e^{-tX}]=k! \bm{\alpha} (t\mathbf{I}-\mathbf{T})^{-k-1}\bm{u}$.

The following example illustrates the CMRS computation in a simple ME setting.
\begin{example}\label{ex:me-two}
Let $X_1\sim\mathrm{Erlang}(2,\lambda)$ and $X_2\sim\mathrm{Exp}(\mu)$ be independent, and set $S=X_1+X_2$.
An ME representation for $X_1$ is
$$
\bm{\alpha}=(1,0),\qquad
\mathbf{T}=
\begin{pmatrix}
-\lambda & \lambda\\
0 & -\lambda
\end{pmatrix},\qquad
\bm{u}=
\begin{pmatrix}
0\\
\lambda
\end{pmatrix},
$$
so that $\mathcal L_{X_1}(t)=\bm{\alpha}(t\mathbf{I}-\mathbf{T})^{-1}\bm{u}=\left(\frac{\lambda}{\lambda+t}\right)^2$.
With $\mathcal L_{X_2}(t)=\frac{\mu}{\mu+t}$, we have
$$
\mathcal L_S(t)=\mathcal L_{X_1}(t)\mathcal L_{X_2}(t)=\frac{\lambda^2\mu}{(\lambda+t)^2(\mu+t)}.
$$
Assuming $\mu\neq\lambda$, Laplace inversion yields
$$
f_S(s)=\frac{\lambda^2\mu}{(\lambda-\mu)^2}\left(e^{-\mu s}-e^{-\lambda s}\left(1+(\lambda-\mu)s\right)\right),\qquad s>0.
$$

For the unconditional expected allocation density of $X_1$, Theorem \ref{thm:main-theorem} and independence give
$$
\mathcal L_1^*(t)=\E[X_1e^{-tS}]
=\E[X_1e^{-tX_1}] \E[e^{-tX_2}]
=\left(-\frac{\d}{\d t}\mathcal L_{X_1}(t)\right)\mathcal L_{X_2}(t)
=\frac{2\lambda^2\mu}{(\lambda+t)^3(\mu+t)}.
$$
	Laplace inversion gives, for $s>0$,
		$$
\xi_1(s)
=\frac{2\lambda^2\mu}{(\lambda-\mu)^3}\left(e^{-\mu s}-e^{-\lambda s}\left(1+(\lambda-\mu)s+\tfrac{(\lambda-\mu)^2}{2}s^2\right)\right).
			$$
	Writing $\delta:=\lambda-\mu$, the CMRS allocation $h_1(s)=\xi_1(s)/f_S(s)$ simplifies to
		$$\E[X_1\mid S=s] =\frac{2}{\delta} \frac{e^{\delta s}-1-\delta s-\tfrac12(\delta s)^2}{e^{\delta s}-1-\delta s}, \qquad s>0.$$
\end{example}
More generally, for independent ME margins, $\mathcal L_S$ and $\mathcal L_i^*$ are rational and Laplace inversion can be performed analytically using partial-fraction decomposition or matrix-exponential representations.

\subsection{Independent compound Katz distributions}\label{subsec:compound-katz}

Compound distributions are commonly used in insurance risk models, where $M_i$ denotes the claim count for policyholder $i$ and $B_{i,k}$ the severity of the $k$th claim. The aggregate loss of policyholder $i$ and the portfolio loss are then
$$X_i:=\sum_{k=1}^{M_i} B_{i,k}, \qquad S=\sum_{i=1}^n X_i.$$ We assume $M_i$ is independent of the i.i.d.\ severities $(B_{i,k})_{k\ge1}$, and the collective risk models $(M_i,(B_{i,k})_{k\ge1})$ are independent across $i=1,\dots,n$. Define the LST of the severity distribution and its negative derivative as
$$\phi_i(t):=\E[e^{-tB_{i,1}}],\qquad \phi_i^{(1)}(t):=-\phi_i'(t)=\E[B_{i,1}e^{-tB_{i,1}}],\qquad t>0.$$

A count distribution $M$ belongs to the Katz $(a,b,0)$ family if its probability generating function $P_M(z):=\E[z^M]$ satisfies
\begin{equation}\label{eq:katz-pgf-derivative}
P_M'(z)=\frac{a+b}{1-az} P_M(z),\qquad |z|<1,
\end{equation}
or equivalently, its probability mass function satisfies the recurrence relation
$$\p(M=k)=\left(a+\frac{b}{k}\right)\p(M=k-1),$$
for $k\ge1$. This family includes the Poisson ($a=0$), binomial ($a<0$), and negative binomial ($a>0$) distributions. In this case, the LST of the compound distribution admits a convenient closed-form expression, and the expected allocation transform $\mathcal L_i^*$ can be expressed in terms of $\mathcal L_S$.

\begin{proposition}\label{prop:edf-frailty-katz}
Suppose the collective risk models $(M_i,(B_{i,k})_{k\ge1})$ are independent across $i=1,\dots,n$, and that each frequency distribution $M_i$ follows a Katz $(a_i,b_i,0)$ distribution. Then, for each $t>0$,
\begin{equation}\label{eq:Li-star-katz}
\mathcal L_i^*(t)
=\frac{a_i+b_i}{1-a_i\phi_i(t)} \phi_i^{(1)}(t) \mathcal L_S(t).
\end{equation}
\end{proposition}

\begin{remark}
Replacing Laplace transforms by ordinary generating functions recovers the compound Katz expected-allocation identities of \citet{blier-wong2025efficient}.
When $a_i=0$ (compound Poisson) and severities share a common distribution, \eqref{eq:Li-star-katz} gives $\mathcal L_i^*(t)\propto \mathcal L_S'(t)$, so $h_i(s)=(\lambda_i/\sum_j\lambda_j)s$, consistent with \cite{denuit2020largeloss}.
\end{remark}

The following example illustrates the application of Proposition \ref{prop:edf-frailty-katz} to a common-shock compound Poisson model. Note that the components of the model are not independent, but conditional independence given the common shock allows for a convenient factorization of the joint LST, after which inversion yields closed-form densities for the aggregate density and expected allocations.

\begin{example}\label{ex:common-shock-cp}
Consider a common-shock compound Poisson model with $n\ge 2$ business lines. Let $N_0\sim\mathrm{Pois}(\lambda_0)$ and $N_i\sim\mathrm{Pois}(\lambda_i)$ for $i=1,\dots,n$ be mutually independent claim counts. Fix weights $p_i\ge 0$ with $\sum_{i=1}^n p_i=1$. Let $(Z_k)_{k\ge 1}$ be i.i.d.\ common-shock severities with $Z_1\sim\mathrm{Exp}(\beta_0)$, and for each $i$ let $(Y_{i,m})_{m\ge 1}$ be i.i.d.\ idiosyncratic severities with $Y_{i,1}\sim\mathrm{Exp}(\beta_i)$, where $\beta_0,\beta_1,\dots,\beta_n>0$ are not necessarily equal. All severity sequences are mutually independent and independent of $(N_0,N_1,\dots,N_n)$. The loss of business line $i$ and the aggregate business lines are
$$X_i:=\sum_{k=1}^{N_0} p_i Z_k+\sum_{m=1}^{N_i} Y_{i,m},\qquad S:=\sum_{i=1}^n X_i.$$
Since $\sum_{i=1}^n p_i=1$, the aggregate simplifies to
$$S=\sum_{k=1}^{N_0} Z_k+\sum_{i=1}^n\sum_{m=1}^{N_i} Y_{i,m},$$
so the law of $S$ does not depend on the weight vector $(p_1,\dots,p_n)$. Using $\mathcal L_Z(u)=\beta_0/(\beta_0+u)$ and $\mathcal L_{Y_i}(u)=\beta_i/(\beta_i+u)$, the joint LST of $\bm X=(X_1,\dots,X_n)$ is
$$\mathcal L_{\bm X}(t_1,\dots,t_n)=\exp\left(\lambda_0\left(\frac{\beta_0}{\beta_0+\sum_{j=1}^n p_j t_j}-1\right)+\sum_{i=1}^n \lambda_i\left(\frac{\beta_i}{\beta_i+t_i}-1\right)\right),$$
for $(t_1,\dots,t_n)\in[0,\infty)^n$. Evaluating on the diagonal gives
$$\mathcal L_S(t)=\mathcal L_{\bm X}(t,\dots,t)=\exp \left(\lambda_0\left(\frac{\beta_0}{\beta_0+t}-1\right)+\sum_{i=1}^n \lambda_i\left(\frac{\beta_i}{\beta_i+t}-1\right)\right),\qquad t\ge 0,$$
and $S$ has a single atom at $0$ of size $\p(S=0)=\exp(-\lambda_S)$, where $\lambda_S:=\lambda_0+\sum_{i=1}^n \lambda_i$. Differentiating $\mathcal L_{\bm X}$ with respect to $t_i$ and evaluating on the diagonal yields the allocation transform
$$\mathcal L_i^*(t)=\E[X_ie^{-tS}]=\mathcal L_S(t)\left(\lambda_0p_i\,\frac{\beta_0}{(\beta_0+t)^2}+\lambda_i\,\frac{\beta_i}{(\beta_i+t)^2}\right),\qquad t>0,$$
and $\nu_i(\{0\})=\E[X_i\id{S=0}]=0$.

To obtain closed-form densities, define for $k=(k_0,k_1,\dots,k_n)\in\mathbb{N}_0^{n+1}$ the rational function
$$R_k(t):=\prod_{j=0}^n\left(\frac{\beta_j}{\beta_j+t}\right)^{k_j},\qquad t\ge 0,$$
and let $r_k$ denote its inverse Laplace transform, which is the density on $(0,\infty)$ of a sum of independent random variables, where for each $j\in\{0,1,\dots,n\}$ with $k_j\ge 1$, the $j$th summand follows a Gamma distribution with shape $k_j$ and rate $\beta_j$. Expanding the compound Poisson probabilities gives, for $s>0$,
$$f_S(s)=e^{-\lambda_S}\sum_{\substack{\bm k\in\mathbb{N}_0^{n+1}\\ k_0+\cdots+k_n\ge 1}}\left(\prod_{j=0}^n \frac{\lambda_j^{k_j}}{k_j!}\right) r_k(s),$$
and the density $\xi_i$ of $\nu_i$ on $(0,\infty)$ is
$$\xi_i(s)=e^{-\lambda_S}\sum_{\bm k\in\mathbb{N}_0^{n+1}}\left(\prod_{j=0}^n \frac{\lambda_j^{k_j}}{k_j!}\right)\left[\frac{\lambda_0p_i}{\beta_0}\, r_{k+2e_0}(s)+\frac{\lambda_i}{\beta_i}\, r_{k+2e_i}(s)\right],\qquad s>0,$$
where $e_j\in\mathbb{R}^{n+1}$ is the $(n+1)$-dimensional unit vector. If $\beta_0,\beta_1,\dots,\beta_n$ are pairwise distinct, each $r_k$ admits the partial-fraction form
$$r_k(s)=\sum_{\substack{j\in\{0,1,\dots,n\}:\\ k_j\ge 1}}\ \sum_{r=1}^{k_j}c_{j,r}(k)\,\frac{s^{r-1}}{(r-1)!}e^{-\beta_j s},\qquad s>0,$$
with coefficients
$$c_{j,r}(k)=\frac{1}{(k_j-r)!}\left.\frac{\d^{\,k_j-r}}{\d t^{\,k_j-r}}\Big[(t+\beta_j)^{k_j}R_k(t)\Big]\right|_{t=-\beta_j}.$$
If some rates coincide, the same representation holds with higher-order Erlang terms coming from repeated poles. On the absolutely continuous part, the CMRS allocation is then
$$h_i(s)=\frac{\xi_i(s)}{f_S(s)},\qquad s>0,\qquad\text{and}\qquad h_i(0)=0.$$
\end{example}

\section{Numerical Laplace inversion for CMRS}\label{sec:numerical}

When analytic inversion is infeasible or the inversion is too complicated to be practical, the CMRS allocation densities must be recovered numerically from their Laplace transforms. This section develops a practical computational framework based on the transform-level representation of CMRS. 

Recall that if $S$ has atoms or singular components, these must be handled separately by removing them from the LST; see Section \ref{subsec:atoms}. The methods described in this section apply to the absolutely continuous part of the CMRS allocation.

\subsection{The Gaver--Stehfest inversion formula}

Let $f:[0,\infty)\to\R$ be locally integrable and of exponential order, with Laplace transform
$$\hat f(t)=\int_0^\infty e^{-ts}f(s)\d s,\qquad t>0.$$
The Gaver--Stehfest \citep{gaver1966observing,stehfest1970algorithm} method produces an approximation to $f(s)$ at a fixed point $s>0$ from finitely many evaluations of $\hat f$ on the positive real axis.
In the notation of \cite{abate2006unified}, the $M$th-order approximation is
\begin{equation}\label{eq:GS-inversion-sec4}
f(s)\approx f^{\mathrm{GS}}_M(s)
:=\frac{\log(2)}{s}\sum_{k=1}^{2M}\zeta_k \hat f \left(\frac{k\log(2)}{s}\right),
\end{equation}
where the weights $\zeta_k$ depend only on $M$ and are given by
\begin{equation}\label{eq:GS-weights-sec4}
\zeta_k = (-1)^{M+k} \sum_{j=\lfloor(k+1)/2\rfloor}^{\min\{k,M\}} \frac{j^{M+1}}{M!}\binom{M}{j}\binom{2j}{j}\binom{j}{k-j},
\qquad k=1,\dots,2M.
\end{equation}
		
In CMRS computations, the same node set $\{k\log(2)/s:1\le k\le 2M\}$ can be used for all inversions at a fixed $s$, allowing $\mathcal L_S$ and all $\mathcal L_i^*$ to be evaluated once and combined.

\subsection{Complex-plane Euler inversion}
\label{subsec:euler-inversion}

In tails of the aggregate distribution, the Gaver--Stehfest method can be unstable due to the rapidly growing and alternating weights $\zeta_k$ and the relatively flat transform values $\hat f(k\log(2)/s)$. This issue is further exacerbated when computing CMRS ratios, since both the numerator and denominator are subject to similar cancellation effects, leading to numerical instability. The Euler method is an alternative Fourier-series method that evaluates the transform at complex points on a vertical line in the complex plane, which can mitigate cancellation issues and improve numerical stability in the tails of distributions. The following lemma justifies the analytic continuation of $\mathcal L_S$ and $\mathcal L_i^*$ at complex arguments, which is required for the Euler method.



\begin{lemma}\label{lem:complex-L}
Let $X_1,\dots,X_n$ be nonnegative random variables on $(\Omega,\mathcal F,\p)$. For $z\in\mathbb C$ with $\Re z>0$, define
$$\mathcal L_S(z):=\E\big[e^{-zS}\big],\qquad \mathcal L_i^*(z):=\E\big[X_i e^{-zS}\big],\qquad i=1,\dots,n.$$
Then the following holds.
\begin{enumerate}[label=\textnormal{(\roman*)}, leftmargin=3.2em]
\item The functions $\mathcal L_S$ and $\mathcal L_i^*$ are well-defined and holomorphic on $\{z\in\mathbb C:\Re z>0\}$.
\item For every integer $m\ge 0$ and every $\gamma>0$, the derivatives satisfy
$$\frac{\d^m}{\d z^m}\mathcal L_S(z)=\E\big[(-S)^m e^{-zS}\big],\qquad \frac{\d^m}{\d z^m}\mathcal L_i^*(z)=\E\big[(-S)^m X_i e^{-zS}\big],$$
for all $z$ with $\Re z\ge \gamma$, and the right-hand sides are finite and locally bounded uniformly on compact subsets of $\{\Re z\ge \gamma\}$.
\item If the joint LST $\mathcal L_{\bm X}(t_1,\dots,t_n)$ extends holomorphically to $\{(t_1,\dots,t_n)\in\mathbb C^n:\Re t_j>0\ \forall j\}$, then for every $z$ with $\Re z>0$,
$$\mathcal L_i^*(z)=-\left.\frac{\partial}{\partial t_i}\mathcal L_{\bm X}(t_1,\dots,t_n)\right|_{t_1=\cdots=t_n=z}.$$
\end{enumerate}
\end{lemma}

A common Laplace inversion strategy is to perform the inversion along a Bromwich line in the complex plane. Let $f$ be of exponential order with Laplace transform $\hat f(z)=\int_0^\infty e^{-zs}f(s)\d s$ defined for $\Re z>\sigma_0$. The Bromwich inversion formula is
$$f(s)=\frac{1}{2\pi i}\int_{\gamma-i\infty}^{\gamma+i\infty}e^{zs}\hat f(z)\d z,\qquad \gamma>\sigma_0.$$
Fourier-series methods approximate this integral by evaluating $\hat f$ at complex points on the vertical line $\Re z=\gamma$. In the Euler method of \cite{abate2006unified}, one chooses $A>0$, sets $\gamma=A/(2s)$, and evaluates at
$$z_k(s)=\frac{A+2\pi i k}{2s},\qquad k=0,1,2,\dots.$$
Define the alternating partial sum with $N$ retained complex terms by
$$S_N(s):=\frac{e^{A/2}}{s}\left(\tfrac12\Re\{\hat f(z_0(s))\}+\sum_{k=1}^{N}(-1)^k\Re\{\hat f(z_k(s))\}\right),$$
Applying Euler acceleration of order $m$ to the sequence $(S_N(s))_{N\ge 0}$ gives the approximation
$$f^{\mathrm{Eul}}_{N,m}(s):=\sum_{r=0}^{m}\binom{m}{r}2^{-m}S_{N+r}(s),$$
where $N$ controls truncation and $m$ controls binomial smoothing; see \cite{abate2006unified} for details. To compute the CMRS, the same complex node set $\{z_k(s):k=0,\dots,N+m\}$ can be used for all inversions at a fixed $s$, allowing $\mathcal L_S$ and all $\mathcal L_i^*$ to be evaluated once.

\subsection{Tuning the inversion}\label{subsec:tuning-inversion}

Both methods involve a trade-off between truncation error and round-off error, and the budget-balance identity \eqref{eq:budget-balance-density} provides a diagnostic that does not depend on the true target functions and can be used to tune the inversion parameters. Inversion parameters should be increased until these identities hold to the required tolerance.

For Gaver--Stehfest, accuracy improves with the order $M$ for smooth targets, but the weights \eqref{eq:GS-weights-sec4} grow rapidly and alternate in sign, so round-off error becomes dominant for large values of $M$. Following \cite{abate2004multi}, achieving $j$ significant digits requires $M\approx 1.1j$ terms computed with approximately $2M$ digits of precision. Double precision restricts $M$ to roughly $8$--$10$; higher accuracy requires extended precision arithmetic (and therefore more computational time). Since $\zeta_k$ involves large combinatorial factors, the binomial coefficients should be assembled in exact or high-precision arithmetic to avoid overflow. The weights depend only on $M$ and are model-free, so they can be computed once and saved for repeated use across different models and $s$-grids.

For the Euler method, the parameters $A$, $N$, and $m$ should be increased until the diagnostic identities stabilize across the $s$-grid \citep{abate2006unified}. 

\section{Tilting for rare-event CMRS}\label{sec:tilting}

\subsection{Esscher tilting and shift identities}

For large values of $s$, the density $f_S(s)$ of the aggregate loss 
and the allocation density $\xi_i(s)$ of the expected allocation measure $\nu_i$ can fall below the representable range of double-precision arithmetic. Numerical inversion then underflows or suffers large relative error, and the ratio $h_i(s)=\xi_i(s)/f_S(s)$ can 
become unstable, see \cite[Section 4]{blier-wong2025efficient} for examples and discussions. Exponential tilting \citep{esscher1932probability} is a standard device when dealing with the computation of tail probabilities and tail expectations on the transform level; see \cite{grubel1999computation,embrechts1993applications, denuit2022mortality,denuit2022conditional} for applications in risk aggregation and risk-sharing. 

Tilting multiplies the density and allocation by a common factor before inversion; this tilt rescales the target functions to mitigate numerical errors. The common factor then cancels in the ratio, so tilting does not change the target CMRS allocation $h_i(s)$ but allows it to be computed accurately. Up-tilting ($\theta > 0$) rescales the tail upwards to mitigate underflow, while down-tilting ($\theta < 0$) rescales the tail downwards to mitigate overflow. In practice, underflow is the dominant issue in tail CMRS computations, so the remainder of this section focuses on $\theta>0$.

\subsection{Tilting in the CMRS setting}

We now apply tilting to reduce tail underflow in the CMRS ratio. For $\theta>0$, define the tilted densities
$$\xi_i^{(\theta)}(s):=e^{\theta s}\xi_i(s),\qquad f_S^{(\theta)}(s):=e^{\theta s}f_S(s).$$
For $\Re t>\theta$,
$$\mathcal L\{\xi_i^{(\theta)}\}(t)=\mathcal L_i^*(t-\theta),\qquad \mathcal L\{f_S^{(\theta)}\}(t)=\mathcal L_S(t-\theta).$$
The CMRS ratio is tilt-invariant, since for $s\ge 0$,
\begin{equation}\label{eq:tilt-invariant-cmrs}
h_i(s)=\frac{\xi_i^{(\theta)}(s)}{f_S^{(\theta)}(s)}=\frac{\xi_i(s)}{f_S(s)}.
\end{equation}

Up-tilting replaces $\mathcal L_S(t)$ and $\mathcal L_i^*(t)$ with $\mathcal L_S(t-\theta)$ and $\mathcal L_i^*(t-\theta)$, where $\theta>0$. Gaver--Stehfest evaluates the transform at a set of real nodes $t_k>0$; with up-tilting, these nodes shift to $t_k-\theta$, which can be negative when $t_k<\theta$. Using Gaver--Stehfest with positive tilting, therefore, requires the transform to be defined and convergent on $(-\theta,\infty)$. For the Laplace--Stieltjes transforms considered here, the domain of convergence is $(0,\infty)$, so this condition fails: Gaver--Stehfest cannot be combined with positive tilting unless the transform admits an analytic continuation to $(-\theta,0]$, which is not the case in general. Complex-plane inversion avoids this difficulty. By choosing a Bromwich contour at $\Re z=\gamma>\theta$, all shifted evaluations $z-\theta$ remain in the right half-plane $\{\Re z>0\}$ where the transforms are analytic. For this reason, tilting is straightforward to implement only with complex-plane methods such as the Euler method, and we will apply tilting to the Euler method in the numerical examples that follow.

\subsection{Numerical implementation}\label{subsec:numerical-implementation}

We illustrate the numerical inversion methods in a setting with exact
expressions for the conditional means, so that the numerical approximations can be compared against the closed-form solution. 

\subsubsection{Common-shock compound Poisson continuation and scaling}\label{subsubsec:cp-continuation-scaling}
We continue Example \ref{ex:common-shock-cp} with a pool of common-shock compound Poisson portfolio with $n = 3$ lines of business. The parameter set is $\lambda_0=1.5$, $(\lambda_1,\lambda_2,\lambda_3)=(0.8,1.1,0.6)$, $\beta_0=0.9$, $(\beta_1,\beta_2,\beta_3)=(1.4,0.7,1.9)$, and $(p_1,p_2,p_3)=(0.2,0.3,0.5)$, evaluated on a grid $s\in[0.1,75]$ with step $0.1$. We use GS inversion ($M=10$) and Euler inversion ($N=25,m=15$) with and without up-tilting ($\theta=0.2$).

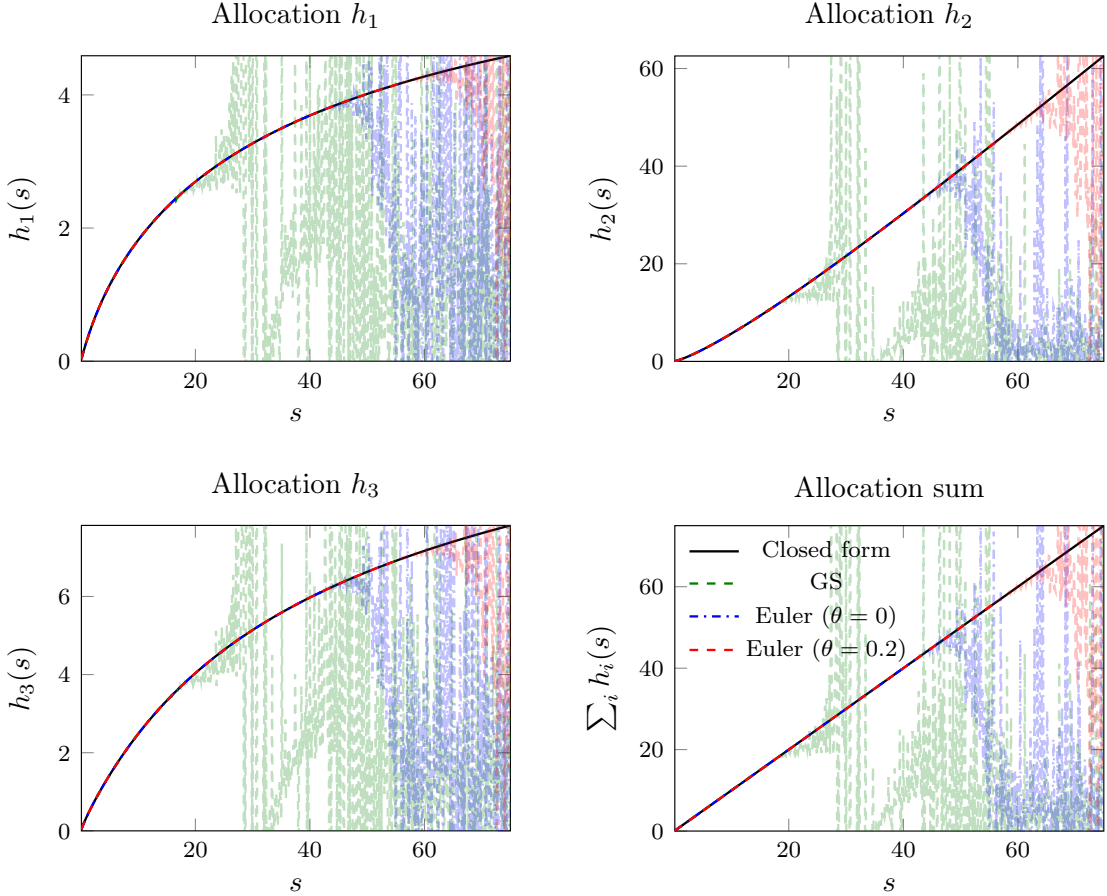
\begin{figure}[!ht]
\centering
\resizebox{0.90\textwidth}{!}{%
\begin{tikzpicture}
\begin{groupplot}[
  group style={group size=2 by 2, horizontal sep=2.0cm, vertical sep=2.0cm},
  width=6.8cm,
  height=5.3cm,
  xmin=0.1, xmax=75,
  xlabel={$s$},
  tick label style={font=\scriptsize},
  label style={font=\small},
  title style={font=\small}
]
\nextgroupplot[
  ylabel={$h_1(s)$},
  title={Allocation $h_1$},
  ymin=0,
  ymax=4.58817061673245
]
\addplot+[black, thick, no marks] table[x=s, y=h_closed_1, col sep=comma]
  {output_common_shock_compound_poisson/allocations_compound_poisson_closed_vs_methods_by_s.csv};
\addplot+[green!50!black, thick, dashed, no marks, restrict x to domain=0.1:17] table[x=s, y=h_gs_1, col sep=comma]
  {output_common_shock_compound_poisson/allocations_compound_poisson_closed_vs_methods_by_s.csv};
\addplot+[green!50!black, thick, dashed, no marks, opacity=0.25, restrict x to domain=17:75, forget plot] table[x=s, y=h_gs_1, col sep=comma]
  {output_common_shock_compound_poisson/allocations_compound_poisson_closed_vs_methods_by_s.csv};
\addplot+[blue, thick, dashdotted, no marks, restrict x to domain=0.1:42] table[x=s, y=h_complex_untilted_1, col sep=comma]
  {output_common_shock_compound_poisson/allocations_compound_poisson_closed_vs_methods_by_s.csv};
\addplot+[blue, thick, dashdotted, no marks, opacity=0.25, restrict x to domain=42:75, forget plot] table[x=s, y=h_complex_untilted_1, col sep=comma]
  {output_common_shock_compound_poisson/allocations_compound_poisson_closed_vs_methods_by_s.csv};
\addplot+[red, thick, dashed, no marks, restrict x to domain=0.1:56] table[x=s, y=h_complex_tilt_1, col sep=comma]
  {output_common_shock_compound_poisson/allocations_compound_poisson_closed_vs_methods_by_s.csv};
\addplot+[red, thick, dashed, no marks, opacity=0.25, restrict x to domain=56:75, forget plot] table[x=s, y=h_complex_tilt_1, col sep=comma]
  {output_common_shock_compound_poisson/allocations_compound_poisson_closed_vs_methods_by_s.csv};

\nextgroupplot[
  ylabel={$h_2(s)$},
  title={Allocation $h_2$},
  ymin=0,
  ymax=62.5787319955666
]
\addplot+[black, thick, no marks] table[x=s, y=h_closed_2, col sep=comma]
  {output_common_shock_compound_poisson/allocations_compound_poisson_closed_vs_methods_by_s.csv};
\addplot+[green!50!black, thick, dashed, no marks, restrict x to domain=0.1:17] table[x=s, y=h_gs_2, col sep=comma]
  {output_common_shock_compound_poisson/allocations_compound_poisson_closed_vs_methods_by_s.csv};
\addplot+[green!50!black, thick, dashed, no marks, opacity=0.25, restrict x to domain=17:75, forget plot] table[x=s, y=h_gs_2, col sep=comma]
  {output_common_shock_compound_poisson/allocations_compound_poisson_closed_vs_methods_by_s.csv};
\addplot+[blue, thick, dashdotted, no marks, restrict x to domain=0.1:42] table[x=s, y=h_complex_untilted_2, col sep=comma]
  {output_common_shock_compound_poisson/allocations_compound_poisson_closed_vs_methods_by_s.csv};
\addplot+[blue, thick, dashdotted, no marks, opacity=0.25, restrict x to domain=42:75, forget plot] table[x=s, y=h_complex_untilted_2, col sep=comma]
  {output_common_shock_compound_poisson/allocations_compound_poisson_closed_vs_methods_by_s.csv};
\addplot+[red, thick, dashed, no marks, restrict x to domain=0.1:56] table[x=s, y=h_complex_tilt_2, col sep=comma]
  {output_common_shock_compound_poisson/allocations_compound_poisson_closed_vs_methods_by_s.csv};
\addplot+[red, thick, dashed, no marks, opacity=0.25, restrict x to domain=56:75, forget plot] table[x=s, y=h_complex_tilt_2, col sep=comma]
  {output_common_shock_compound_poisson/allocations_compound_poisson_closed_vs_methods_by_s.csv};

\nextgroupplot[
  ylabel={$h_3(s)$},
  title={Allocation $h_3$},
  ymin=0,
  ymax=7.81871139181667
]
\addplot+[black, thick, no marks] table[x=s, y=h_closed_3, col sep=comma]
  {output_common_shock_compound_poisson/allocations_compound_poisson_closed_vs_methods_by_s.csv};
\addplot+[green!50!black, thick, dashed, no marks, restrict x to domain=0.1:17] table[x=s, y=h_gs_3, col sep=comma]
  {output_common_shock_compound_poisson/allocations_compound_poisson_closed_vs_methods_by_s.csv};
\addplot+[green!50!black, thick, dashed, no marks, opacity=0.25, restrict x to domain=17:75, forget plot] table[x=s, y=h_gs_3, col sep=comma]
  {output_common_shock_compound_poisson/allocations_compound_poisson_closed_vs_methods_by_s.csv};
\addplot+[blue, thick, dashdotted, no marks, restrict x to domain=0.1:42] table[x=s, y=h_complex_untilted_3, col sep=comma]
  {output_common_shock_compound_poisson/allocations_compound_poisson_closed_vs_methods_by_s.csv};
\addplot+[blue, thick, dashdotted, no marks, opacity=0.25, restrict x to domain=42:75, forget plot] table[x=s, y=h_complex_untilted_3, col sep=comma]
  {output_common_shock_compound_poisson/allocations_compound_poisson_closed_vs_methods_by_s.csv};
\addplot+[red, thick, dashed, no marks, restrict x to domain=0.1:56] table[x=s, y=h_complex_tilt_3, col sep=comma]
  {output_common_shock_compound_poisson/allocations_compound_poisson_closed_vs_methods_by_s.csv};
\addplot+[red, thick, dashed, no marks, opacity=0.25, restrict x to domain=56:75, forget plot] table[x=s, y=h_complex_tilt_3, col sep=comma]
  {output_common_shock_compound_poisson/allocations_compound_poisson_closed_vs_methods_by_s.csv};

\nextgroupplot[
  ylabel={$\sum_i h_i(s)$},
  title={Allocation sum},
  ymin=0,
  ymax=75,
  legend style={font=\scriptsize, draw=none, fill=none, at={(0.01,0.99)}, anchor=north west}
]
\addplot+[black, thick, no marks] table[x=s, y=sum_h_closed, col sep=comma]
  {output_common_shock_compound_poisson/allocations_compound_poisson_closed_vs_methods_by_s.csv};
\addplot+[green!50!black, thick, dashed, no marks, restrict x to domain=0.1:17] table[x=s, y=sum_h_gs, col sep=comma]
  {output_common_shock_compound_poisson/allocations_compound_poisson_closed_vs_methods_by_s.csv};
\addplot+[green!50!black, thick, dashed, no marks, opacity=0.25, restrict x to domain=17:75, forget plot] table[x=s, y=sum_h_gs, col sep=comma]
  {output_common_shock_compound_poisson/allocations_compound_poisson_closed_vs_methods_by_s.csv};
\addplot+[blue, thick, dashdotted, no marks, restrict x to domain=0.1:42] table[x=s, y=sum_h_complex_untilted, col sep=comma]
  {output_common_shock_compound_poisson/allocations_compound_poisson_closed_vs_methods_by_s.csv};
\addplot+[blue, thick, dashdotted, no marks, opacity=0.25, restrict x to domain=42:75, forget plot] table[x=s, y=sum_h_complex_untilted, col sep=comma]
  {output_common_shock_compound_poisson/allocations_compound_poisson_closed_vs_methods_by_s.csv};
\addplot+[red, thick, dashed, no marks, restrict x to domain=0.1:56] table[x=s, y=sum_h_complex_tilt, col sep=comma]
  {output_common_shock_compound_poisson/allocations_compound_poisson_closed_vs_methods_by_s.csv};
\addplot+[red, thick, dashed, no marks, opacity=0.25, restrict x to domain=56:75, forget plot] table[x=s, y=sum_h_complex_tilt, col sep=comma]
  {output_common_shock_compound_poisson/allocations_compound_poisson_closed_vs_methods_by_s.csv};
\legend{Closed form, GS, Euler ($\theta=0$), Euler ($\theta=0.2$)}
\end{groupplot}
\end{tikzpicture}
}
\caption{Common-shock compound Poisson: closed-form benchmark, GS, and Euler inversion with $\theta=0.2$ and $\theta=0$. Note that we show the curves as faded once the allocation-sum diagnostic breaks down, which occurs at different values of $s$ across methods. }
\label{fig:compound-poisson-complex-inversion}
\end{figure}

Figure \ref{fig:compound-poisson-complex-inversion} shows the allocation curves and the allocation-sum diagnostic for the three methods. The GS method performs well for small values of $s$ but breaks down in the tail, where the allocation sum deviates significantly from $s$. The Euler method also suffers from underflow in the tail, with the untilted version breaking down earlier than the tilted version. 

The baseline runtime comparison for this $n=3$ setting is shown in Table \ref{tab:compound-poisson-runtime}. While the exact method is the most accurate, it is also the slowest by a wide margin, as it requires evaluating an infinite sum (truncated to a large number of terms) for each $s$. A large truncation point is needed to capture the tail behaviour, which is important for accurate CMRS allocations at large $s$. The computation is exponential in $n$ because the number of terms grows exponentially with $n$, making the exact method infeasible for larger portfolios. The GS and Euler methods are much faster, with the GS method being slightly faster in this example. Tilting adds a small overhead to the Euler method.
\begin{table}[!ht]
\centering
\small
\begin{tabular}{lr}
\toprule
Method & Runtime (s) \\
\midrule
Closed-form series       & 27.643 \\
GS inversion             & 0.0440 \\
Euler inversion ($\theta=0$) & 0.0880 \\
Euler inversion ($\theta=0.2$) & 0.1050 \\
\bottomrule
\end{tabular}
\caption{Runtime summary for the common-shock compound Poisson with $n = 3$.}
\label{tab:compound-poisson-runtime}
\end{table}

Figure \ref{fig:compound-poisson-runtime-scaling} compares the runtime scaling of the GS and Euler methods as the number of participants $n$ increases. The runtime is measured on a grid of $s$ values, and the mean runtime from 100 replications is plotted against $n$ on a log-log scale. The GS method scales better than the Euler method by a small constant factor, but both methods exhibit similar scaling behaviour as $n$ increases. The tilting parameter $\theta=0.2$ does not significantly affect the runtime compared to the untilted version, indicating that the overhead from tilting is negligible in this setting. Note that we do not include the closed-form method in this scaling comparison because it becomes infeasible for larger $n$ due to its exponential runtime (even for $n = 5$). Although a closed-form benchmark is unavailable beyond small $n$, the budget-balance diagnostic \eqref{eq:budget-balance-density} remains computable at any dimension and provides a model-free accuracy check. In our experiments, with adequate tuning of the inversion parameters as described in Section \ref{subsec:tuning-inversion}, the identity $\sum_i \xi_i(s)\approx s f_S(s)$ continues to hold to high precision in the body of the distribution of $S$ across the full range of $n$ tested, indicating that the numerical approximations remain accurate as $n$ increases.

\begin{figure}[!ht]
\centering
\resizebox{0.75\textwidth}{!}{%
\begin{tikzpicture}
\begin{axis}[
  width=12.0cm,
  height=7.0cm,
  xmode=log,
  ymode=log,
  log basis x=10,
  log basis y=10,
  xmin=4, xmax=120000,
  ymin=0.08, ymax=700,
  xlabel={Number of participants $n$ (log scale)},
  ylabel={Mean runtime (seconds)},
  xtick={5,10,20,50,100,1000,10000,100000},
  xticklabels={5,10,20,50,100,1000,10000,100000},
  ytick={0.1,1,10,100,500},
  yticklabels={0.1,1,10,100,500},
  scaled x ticks=false,
  scaled y ticks=false,
  tick label style={font=\scriptsize},
  label style={font=\small},
  legend pos=north west,
  legend style={font=\scriptsize, draw=none, fill=none}
]
\addplot+[green!50!black, thick, mark=o, mark size=1.8pt] coordinates {
  (5,0.0990)
  (10,0.1086)
  (20,0.1187)
  (50,0.1700)
  (100,0.2435)
  (1000,1.5848)
  (10000,14.7500)
  (100000,171.8400)
};
\addplot+[blue, thick, dashdotted, mark=square*, mark size=1.8pt, mark options={draw=blue, fill=blue}] coordinates {
  (5,0.1352)
  (10,0.1515)
  (20,0.2004)
  (50,0.3099)
  (100,0.5177)
  (1000,3.8738)
  (10000,42.6825)
  (100000,496.5750)
};
\addplot+[red, thick, dashed, mark=triangle*, mark size=2.0pt, mark options={draw=red, fill=red}] coordinates {
  (5,0.1325)
  (10,0.1668)
  (20,0.1985)
  (50,0.3131)
  (100,0.5162)
  (1000,4.0035)
  (10000,40.1275)
  (100000,477.0900)
};
\legend{GS untilted, Euler ($\theta=0$), Euler ($\theta=0.2$)}
\end{axis}
\end{tikzpicture}
}
\caption{Scaling of full-grid mean runtime (seconds) under the same runtime basis as Table \ref{tab:compound-poisson-runtime}.}
\label{fig:compound-poisson-runtime-scaling}
\end{figure}
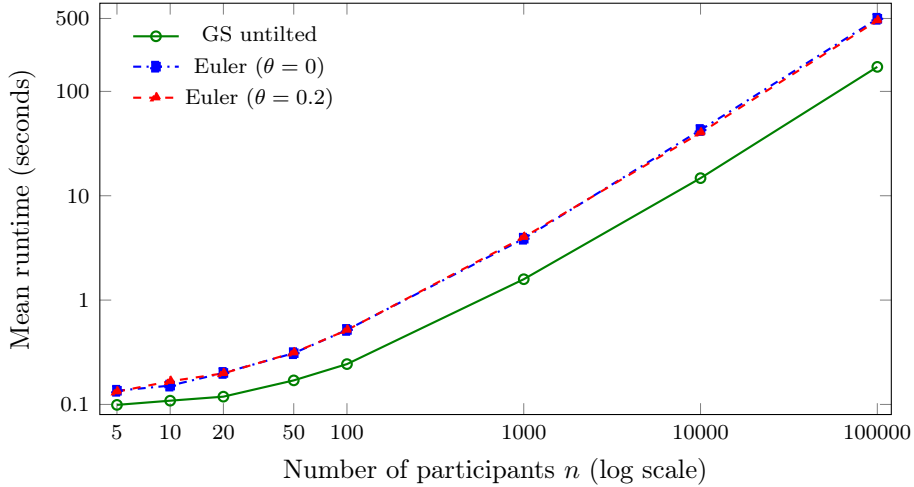

\subsubsection{Independent lognormal implementation}\label{ex:indep-lognormal-implementation}
Consider $n=3$ independent lognormal risks $X_i\sim\mathrm{LN}(\mu_i,\sigma_i^2)$, with parameters set such that the marginal means are $(1,2,2)$ and variances are $(5,2,5)$, with aggregate loss $S=\sum_{i=1}^3 X_i$.
For $t>0$, set
$$\mathcal L_S(t)=\prod_{j=1}^3 \E[e^{-tX_j}],\qquad \mathcal L_i^*(t)=\E[X_i e^{-tX_i}] \prod_{j\neq i}\E[e^{-tX_j}].$$
In contrast to Example \ref{ex:common-shock-cp}, where an exact closed-form benchmark is available, no closed-form benchmark exists here; we compute the CMRS numerically with the Gauss--Hermite quadrature and Laplace inversion as in \citep{furman2020lognormal}.

\begin{figure}[!ht]
\centering
\resizebox{0.75\textwidth}{!}{%
\begin{tikzpicture}
\begin{groupplot}[
  group style={group size=2 by 2, horizontal sep=2.0cm, vertical sep=2.4cm},
  width=5.8cm,
  height=5.8cm,
  xmin=0.5, xmax=25,
  xlabel={$s$},
  grid=both,
  tick label style={font=\scriptsize},
  label style={font=\small},
  title style={font=\small}
]
\nextgroupplot[
  ylabel={$h_1(s)$},
  title={Allocation $h_1$},
  legend pos=north west,
  legend style={font=\scriptsize, draw=none, fill=none}
]
\addplot+[green!50!black, thick, no marks] table[x=s, y=h_gs_1, col sep=comma]
  {output_lognormal/allocations_lognormal_complex_by_s.csv};
\addplot+[blue, thick, dashdotted, no marks] table[x=s, y=h_complex_untilted_1, col sep=comma]
  {output_lognormal/allocations_lognormal_complex_by_s.csv};
\addplot+[red, thick, dashed, no marks] table[x=s, y=h_complex_tilt_1, col sep=comma]
  {output_lognormal/allocations_lognormal_complex_by_s.csv};
\legend{GS, Euler ($\theta=0$), Euler ($\theta=0.8$)}

\nextgroupplot[
  ylabel={$h_2(s)$},
  title={Allocation $h_2$}
]
\addplot+[green!50!black, thick, no marks] table[x=s, y=h_gs_2, col sep=comma]
  {output_lognormal/allocations_lognormal_complex_by_s.csv};
\addplot+[blue, thick, dashdotted, no marks] table[x=s, y=h_complex_untilted_2, col sep=comma]
  {output_lognormal/allocations_lognormal_complex_by_s.csv};
\addplot+[red, thick, dashed, no marks] table[x=s, y=h_complex_tilt_2, col sep=comma]
  {output_lognormal/allocations_lognormal_complex_by_s.csv};

\nextgroupplot[
  ylabel={$h_3(s)$},
  title={Allocation $h_3$}
]
\addplot+[green!50!black, thick, no marks] table[x=s, y=h_gs_3, col sep=comma]
  {output_lognormal/allocations_lognormal_complex_by_s.csv};
\addplot+[blue, thick, dashdotted, no marks] table[x=s, y=h_complex_untilted_3, col sep=comma]
  {output_lognormal/allocations_lognormal_complex_by_s.csv};
\addplot+[red, thick, dashed, no marks] table[x=s, y=h_complex_tilt_3, col sep=comma]
  {output_lognormal/allocations_lognormal_complex_by_s.csv};

\nextgroupplot[
  ylabel={$\sum_i h_i(s)$},
  title={Allocation Sum},
  legend pos=north west,
  legend style={font=\scriptsize, draw=none, fill=none}
]
\addplot+[green!50!black, thick, no marks] table[x=s, y=sum_h_gs, col sep=comma]
  {output_lognormal/allocations_lognormal_complex_by_s.csv};
\addplot+[blue, thick, dashdotted, no marks] table[x=s, y=sum_h_complex_untilted, col sep=comma]
  {output_lognormal/allocations_lognormal_complex_by_s.csv};
\addplot+[red, thick, dashed, no marks] table[x=s, y=sum_h_complex_tilt, col sep=comma]
  {output_lognormal/allocations_lognormal_complex_by_s.csv};
\addplot+[gray!70!black, thick, dotted, no marks] table[x=s, y=s, col sep=comma]
  {output_lognormal/allocations_lognormal_complex_by_s.csv};
\legend{GS, Euler ($\theta=0$), Euler ($\theta=0.8$), $y=s$}
\end{groupplot}
\end{tikzpicture}
}
\caption{Independent lognormals, $s\in[0.5,25]$: GS (untilted) versus Euler inversion with $\theta=0.8$ and $\theta=0$. Panels show the three component allocations $h_i(s)$ and the allocation-sum diagnostic $\sum_i h_i(s)$ against the identity $y=s$.}
\label{fig:lognormal-complex-inversion}
\end{figure}
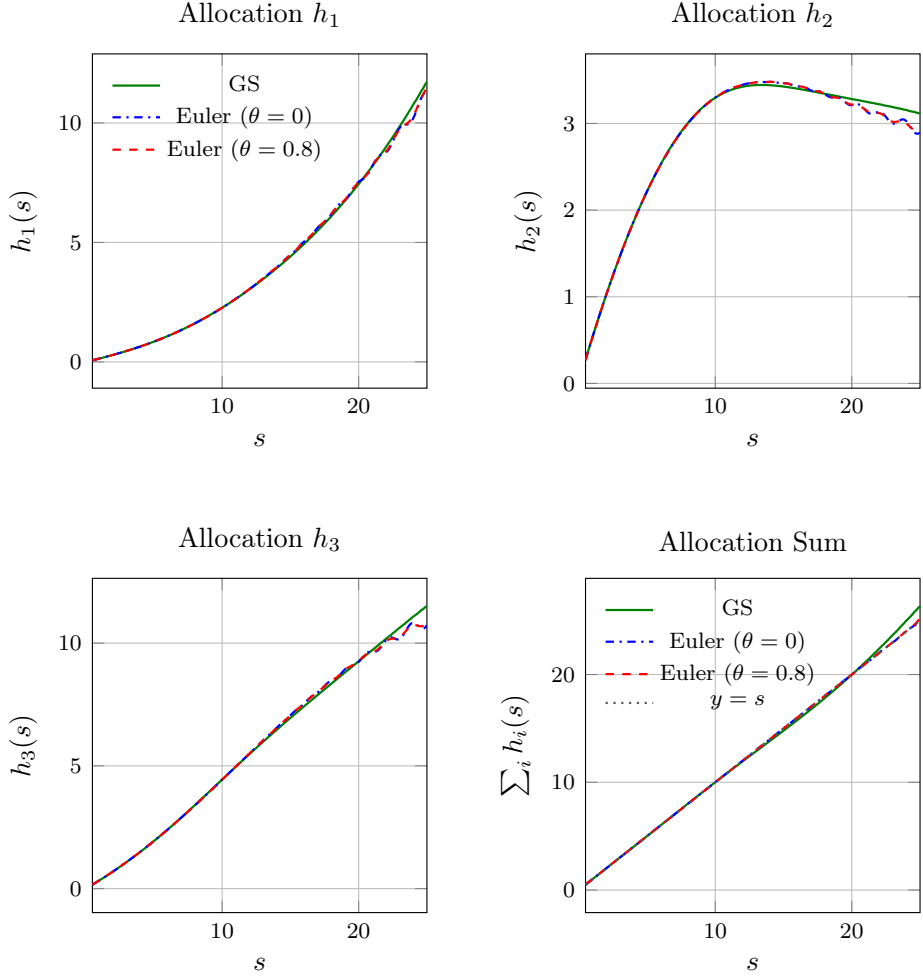

Figure \ref{fig:lognormal-complex-inversion} compares GS (untilted), Euler (untilted), and Euler with up-tilting. In the body of the distribution, all three methods agree closely. In the tail, the GS method remains smooth, but the sum of allocations deviates from the budget-balance identity, indicating numerical instability. The untilted and tilted Euler methods both exhibit some deterioration in the tail; in this case, tilting slightly improves stability, but the effect is less pronounced than in the exponential case. 

\section{Conclusion}\label{sec:conclusion}

In this paper, we develop a general framework for computing conditional mean risk-sharing (CMRS) allocations in aggregate loss models. The key insight is that the CMRS allocation $h_i(s)=\E[X_i\mid S=s]$ can be expressed in terms of the Laplace transforms $\mathcal L_S(t)=\E[e^{-tS}]$ and $\mathcal L_i^*(t)=\E[X_i e^{-tS}]$.

The main identity $\mathcal L_i^*(t)=-\partial_{t_i}\mathcal L_{\bm X}(t,\dots,t)$ reduces CMRS computation to evaluating $\mathcal L_S$ and $\mathcal L_i^*$ and performing one-dimensional Laplace inversions. The allocation-measure formulation $\nu_i(B)=\E[X_i\id{S\in B}]$ provides the rigorous foundation for the CMRS definition, characterizes the conditional means as Radon--Nikodym derivatives, and yields the diagnostic identity $\sum_i\xi_i(s)=sf_S(s)$, a model-independent check for numerical implementations.

For families with rational LSTs, including mixed exponentials and matrix-exponential margins, this yields closed-form CMRS expressions. For the general case, the approach is to compute $\mathcal L_S$ and $\mathcal L_i^*$ on a suitable grid of $t$-values, apply tilting to improve numerical stability, and then perform the numerical inversion using Gaver--Stehfest or Euler methods. The approach applies whenever $\mathcal L_S$ and $\mathcal L_i^*$ can be evaluated, including lognormal portfolios (via quadrature \cite{furman2020lognormal}) and models admitting saddlepoint or Padé approximations \citep{ramsay2008distribution,asmussen2016laplace}.

We concede that Theorem \ref{thm:main-theorem} is, at its core, a straightforward consequence of differentiation under the integral sign and the Radon--Nikodym theorem; neither the identity $L_i^*(t) = -\partial_{t_i} L_{\mathbf{X}}(t,\ldots,t)$ nor the measure-theoretic formulation of CMRS requires deep probabilistic machinery. The principal contribution of this paper is the computational framework it enables. By reducing CMRS evaluation to one-dimensional Laplace inversions of $L_S$ and $L_i^*$, the approach eliminates the multidimensional integration that makes direct density-based computation infeasible beyond small $n$. This shift from a density-level calculation to a transform-level one is what permits scaling to large portfolios: the examples in Section \ref{subsec:numerical-implementation} demonstrate that portfolios with $n = 100\,000$ participants can be handled routinely, whereas the closed-form series expansion is already impractical for $n = 5$, even when explicit formulas are available for the CMRS allocations. 

The conditional mean $s\mapsto h_i(s)=\E[X_i\mid S=s]$ also connects directly to capital allocation. The Euler-based allocation principle is a widely used approach to capital allocation. For a positively homogeneous risk measure $\rho(S)$ and a decomposition $S=\sum_i X_i$, the contribution of line $i$ is the directional derivative
$$C_i^{\rho}:=\left.\frac{\partial}{\partial\varepsilon}\rho \left(S+\varepsilon X_i\right)\right|_{\varepsilon=0},\qquad \sum_{i=1}^n C_i^{\rho}=\rho(S),$$
whenever the derivative exists. For quantile-based capital $\rho(S)=\mathrm{VaR}_q(S)$, and in the continuous case, the Euler contribution reduces to a conditional mean at the capital threshold:
$$C_i^{\mathrm{VaR}_q}=\E \left[X_i\mid S=\mathrm{VaR}_q(S)\right]=h_i \left(\mathrm{VaR}_q(S)\right),$$
which we are already equipped to compute using the transform-based approach developed in this paper. For tail-based capital such as the tail value-at-risk $\mathrm{TVaR}_q$, the Euler contribution is the tail-conditional mean
$$C_i^{\mathrm{TVaR}_q}=\E \left[X_i\mid S\ge \mathrm{VaR}_q(S)\right]=\frac{1}{\p \left(S\ge \mathrm{VaR}_q(S)\right)}\int_{\mathrm{VaR}_q(S)}^{\infty} h_i(s) f_S(s)\d s,$$
when $S$ has a density $f_S$. The integrand is exactly the allocation measure density $\xi_i(s)=h_i(s)f_S(s)$.

\section*{Acknowledgements}

The author acknowledges financial support from the Natural Sciences and Engineering Research Council of Canada (RGPIN-2025-06879).

\appendix

\section{Omitted proofs}\label{sec:proofs}

\begin{proof}[Proof of Theorem \ref{thm:main-theorem}]
  For part (i), fix $(t_1,\dots,t_n)\in[0,\infty)^n$. If $t_i>0$, take an $h$ such that $|h|<t_i/2$, then $t_i+h>0$. It follows that
$$\frac{\mathcal L_{\bm X}(t_1,\dots,t_i+h,\dots,t_n)-\mathcal L_{\bm X}(t_1,\dots,t_i,\dots,t_n)}{h} =\E\left[e^{-\sum_{j=1}^n t_j X_j}\frac{e^{-hX_i}-1}{h}\right].$$
We further have that $\frac{e^{-hx}-1}{h}=-x\int_0^1 e^{-hux}\d u$, from which we obtain the bound
$$\left|e^{-t_i x}\frac{e^{-hx}-1}{h}\right|\le x e^{-(t_i-|h|)x}\le x e^{-(t_i/2)x}\le \frac{2}{e t_i}.$$
By the dominated convergence theorem \citep[Chapter 13]{feller1957introduction}, the derivative in \eqref{eq:partial-derivative} holds and exists two-sided. If $t_i=0$, take $h>0$ and note that $(1-e^{-hx})/h\uparrow x$ as $h\downarrow 0$ for each $x\ge 0$.
By the monotone convergence theorem, the right derivative exists (possibly $-\infty$) and equals the right-hand side of \eqref{eq:partial-derivative} with $t_i=0$. At $(t_1,\dots,t_n)=(0,\dots,0)$ the derivative is finite if and only if $\E[X_i]<\infty$. Evaluating on the diagonal $t_1=\cdots=t_n=t$ gives \eqref{eq:Li-star-expectation}; at $t=0$ this leads to $\mathcal L_i^*(0)=\E[X_i]$, which may be infinite. For $t>0$, we have
$X_i e^{-tS}\le Se^{-tS}$ and $\sup_{x\ge 0} x e^{-t x}<\infty$ so $\mathcal L_i^*(t)<\infty$ for each $t>0$.

For part (ii), first note that the measure $\nu_i$ is $\sigma$-finite and absolutely continuous with respect to $\mu_S$. To show this, consider the case $\mu_S(B)=0$, then $\id{S\in B}=0$ almost surely, hence $\nu_i(B)=\E[X_i\id{S\in B}]=0$ and therefore $\nu_i\ll \mu_S$.
By the Radon--Nikodym theorem \citep[Section 32]{billingsley1995probability}, there exists a measurable $m_i:[0,\infty)\to[0,\infty)$ such that
\begin{equation}\label{eq:nu-rn}
\nu_i(B)=\int_{[0,\infty)} \id{s\in B}m_i(s)\mu_S(\d s),\qquad B\in\mathcal B([0,\infty)).
\end{equation}
For any bounded nonnegative measurable function $\varphi$,
$$ \E[\varphi(S)X_i]=\int_{[0,\infty)} \varphi(s)\nu_i(\d s) =\int_{[0,\infty)} \varphi(s)m_i(s)\mu_S(\d s) =\E[\varphi(S)m_i(S)].$$
This characterizes $\E[X_i\mid \sigma(S)]$ for $X_i\ge 0$. If in addition $X_i\in L^1$, the same identity holds for bounded signed measurable $\varphi$. In particular, since $0\le X_i\le S$, we have $0\le m_i(s)\le s$ for $\mu_S$-almost every $s$; hence the conditional mean is finite pointwise even when $\E[X_i]=\infty$.

For part (iii), if $\mu_S(\d s)=f_S(s) \d s$, then \eqref{eq:nu-rn} becomes
	$\nu_i(B)=\int_B m_i(s)f_S(s) \d s$, so $\xi_i(s) = m_i(s)f_S(s)$ for Lebesgue-almost every $s$, which corresponds to \eqref{eq:xi-factorization}. Then, \eqref{eq:Li-star-laplace} follows from \eqref{eq:Li-star-expectation} and the definition of $\nu_i$.
\end{proof}

\begin{proof}[Proof of Corollary \ref{cor:convolution-independent}]
By independence and the existence of densities, $(X_i,S_{-i})$ has joint density
$f_{X_i}(x)f_{S_{-i}}(y)$ on $(0,\infty)^2$. For $s>0$, we have
$$f_S(s)= \int_0^s f_{X_i}(x)f_{S_{-i}}(s-x)\d x= (f_{X_i}*f_{S_{-i}})(s).$$
Since $\nu_i(B)=\E[X_i\id{S\in B}]$ has density $\xi_i$, we have
$$\nu_i(B) =\iint_{(0,\infty)^2} x\id{x+y\in B} f_{X_i}(x)f_{S_{-i}}(y)\d x\d y =\int_B\left(\int_0^s x f_{X_i}(x)f_{S_{-i}}(s-x)\d x\right)\d s.$$
Hence $\nu_i(\d s)=\xi_i(s)\d s$ with
$$\xi_i(s)=\int_0^s x f_{X_i}(x)f_{S_{-i}}(s-x)\d x =\left(x f_{X_i}(x)\right)*f_{S_{-i}}(s),$$
as claimed.
\end{proof}

\begin{proof}[Proof of Lemma \ref{lem:absolute-continuity}]
		The relations $\nu_i\ll \mu_S$ and $\nu\ll\mu_S$ follow as in Theorem \ref{thm:main-theorem}(ii). For $\nu_i\ll\nu$, note that if $\nu(B)=0$ then $S\id{S\in B}=0$ almost surely, and since $0\le X_i\le S$ we also have $X_i\id{S\in B}=0$ almost surely, and hence $\nu_i(B)=0$. For the chain rule, we have $\nu=\sum_{j=1}^n \nu_j$ and, by Theorem \ref{thm:main-theorem}(ii), 
	$\frac{\d\nu_j}{\d\mu_S}(s)=\E[X_j\mid S=s]$ for $\mu_S$-almost every $s$.
	Therefore $\frac{\d\nu}{\d\mu_S}(s)=\sum_j \E[X_j\mid S=s]=\E[S\mid S=s]=s$ for $\mu_S$-almost every $s>0$.
Since $\nu_i\ll\nu\ll\mu_S$, the Radon--Nikodym chain rule (see, for instance, Section 32 of \cite{billingsley1995probability}) yields \eqref{eq:rn-chain}.
\end{proof}

\begin{proof}[Proof of Proposition \ref{prop:frailty-Lstar}]
We first verify the joint measurability assumption stated in Section~\ref{subsec:frailty-framework}. For each fixed $t\ge 0$, the map $\theta\mapsto \mathcal L_{X_i\mid\Theta=\theta}(t)$ is Borel measurable whenever the conditional distributions arise from a regular conditional distribution of $(X_1,\dots,X_n)$ given $\Theta$. Since $t\mapsto \mathcal L_{X_i\mid\Theta=\theta}(t)$ is continuous for every $\theta$ by dominated convergence (using $e^{-tX_i}\le 1$), the map $(\theta,t)\mapsto \mathcal L_{X_i\mid\Theta=\theta}(t)$ is jointly measurable. Moreover, the uniform bound $|\mathcal L_{X_i\mid\Theta=\theta}(t)|\le 1$ provides a $\mu_\Theta$-integrable dominating function, so finite products of conditional LSTs may be interchanged with the outer expectation $\E_\Theta[\cdot]$ by dominated convergence.

Note that $0\le X_i e^{-tS}\le 1/(et)$ ensures that $\E[X_i e^{-tS}]<\infty$ for each $t>0$. By the tower property, we have $L_i^*(t)=\mathbb{E}\left[\mathbb{E}[X_i e^{-tS}\mid \Theta]\right].$ Conditional independence on $\Theta$ yields
$$\mathbb{E}[X_i e^{-tS}\mid \Theta] = \mathbb{E}[X_i e^{-tX_i}\mid \Theta] \prod_{j\ne i}\mathbb{E}[e^{-tX_j}\mid \Theta],$$
and taking outer expectations proves the claim.
\end{proof}

\begin{proof}[Proof of Proposition \ref{prop:edf-frailty-LST}]
The identity \eqref{eq:edf-frailty-sum} follows by inserting \eqref{eq:edf-LST} into \eqref{eq:frailty-sum-LST}.
For \eqref{eq:edf-frailty-Li}, start from Proposition \ref{prop:frailty-Lstar} and substitute \eqref{eq:edf-LST} for each conditional LST. Using \eqref{eq:edf-derivative}, we have
$\E[X_i e^{-tX_i}\mid \Theta=\theta]=-\partial_t\mathcal L_{X_i\mid\Theta=\theta}(t)=\kappa_i'(\eta_i(\theta)-\phi_i t)\mathcal L_{X_i\mid\Theta=\theta}(t)$,
which yields \eqref{eq:edf-frailty-Li}.
\end{proof}

\begin{proof}[Proof of Proposition \ref{thm:mex-arch}]
  Write $r_i:=\theta/\lambda_i$. Conditional on $\Theta=\theta>0$, we have 
$\mathcal L_{X_i\mid\Theta=\theta}(t)=r_i/(r_i+t)$ and $\mathcal L_{S\mid\Theta=\theta}(t)=\prod_{j=1}^n\frac{r_j}{r_j+t}.$ If the $r_j$ are distinct, partial fractions yield
$$\mathcal L_{S\mid\Theta=\theta}(t)= \sum_{k=1}^n \frac{A_k(\bm r)r_k}{r_k+t},\qquad A_k(\bm r):=\prod_{m\neq k}\frac{r_m}{r_m-r_k},$$
so the conditional density is
$$f_{S\mid\Theta=\theta}(s)=\sum_{k=1}^n A_k(\bm r)r_k e^{-r_k s},\qquad s>0.$$

By the derivative identity evaluated at $t_1=\cdots=t_n$ applied under the conditional law,
\begin{align*}
\mathcal L_{i\mid\Theta=\theta}^*(t)
&:=\E[X_i e^{-tS}\mid\Theta=\theta]
 = -\left.\frac{\partial}{\partial t_i}\mathcal L_{\bm X\mid\Theta=\theta}(t_1,\dots,t_n)\right|_{t_1=\cdots=t_n=t}\\
&= \frac{r_i}{(r_i+t)^2}\prod_{j\neq i}\frac{r_j}{r_j+t}
 = \frac{\mathcal L_{S\mid\Theta=\theta}(t)}{r_i+t}.
\end{align*}
Thus $\mathcal L_{i\mid\Theta=\theta}^*(t)=\mathcal L\{\xi_{i\mid\Theta=\theta}\}(t)$ with
$$\xi_{i\mid\Theta=\theta}(s)=\int_0^s f_{S\mid\Theta=\theta}(u)e^{-r_i(s-u)}\d u.$$
Inserting the mixture representation of $f_{S\mid\Theta=\theta}$ and evaluating the elementary integrals yields
\begin{align*}
\xi_{i\mid\Theta=\theta}(s)
&= A_i(\bm r)r_i s e^{-r_i s}
 + \sum_{k\neq i}A_k(\bm r)r_k
   e^{-r_i s}\frac{1-e^{-(r_k-r_i)s}}{r_k-r_i}.
\end{align*}

		Next we derive the unconditional versions. Since $r_j=\Theta/\lambda_j$, we have $A_k(\bm r)=A_k(\bm\lambda)$ and
		$r_k/(r_k-r_i)=\lambda_i/(\lambda_i-\lambda_k)$. Since the coefficients $A_k(\bm\lambda)$ may have alternating signs, we apply the dominated convergence theorem to justify interchanging expectation with the finite signed sum: for each fixed $s>0$,
$$\left|\xi_{i\mid\Theta}(s)\right| \le s\sum_{k=1}^n |A_k(\bm\lambda)| \frac{\Theta}{\lambda_k}e^{-(s/\lambda_k)\Theta},$$
and since $\Theta e^{-(s/\lambda_k)\Theta}\le \lambda_k/(es)$ deterministically, the right-hand side has finite expectation.
Therefore, expectation may be interchanged with the finite signed sum. Taking expectations and using the identities $ \E[e^{-(s/\lambda)\Theta}]=\mathcal L_\Theta(s/\lambda)$ and $\E[\Theta e^{-(s/\lambda)\Theta}]= -\mathcal L_\Theta'(s/\lambda),$ we obtain the stated formulas for $f_S$ and $\xi_i$. Note that since $\nu_i(B)=\E[X_i\id{S\in B}]=\int_B \xi_i(s)\d s$ and $\mu_S(\d s)=f_S(s)\d s$, the Radon--Nikodym derivative yields $\E[X_i\mid S=s]=\xi_i(s)/f_S(s)$ for $\mu_S$-almost every $s>0$ with $f_S(s)>0$.

For the homogeneous case, let $\lambda_1=\cdots=\lambda_n=\lambda$.
Conditional on $\Theta=\theta>0$, the variables $X_1,\dots,X_n$ are exchangeable, so
$\E[X_i\mid S,\Theta=\theta]=\E[X_j\mid S,\Theta=\theta]$ for all $i,j$.
Summing over $i$ gives $S=\E[S\mid S,\Theta=\theta]=n \E[X_i\mid S,\Theta=\theta]$, hence
$\E[X_i\mid S,\Theta=\theta]=S/n$ almost surely.
Therefore $\E[X_i\mid S]=S/n$ almost surely, that is $\E[X_i\mid S=s]=s/n$ for $\mu_S$-almost every $s>0$. If $S$ has a density $f_S$, then $\xi_i(s)=(s/n)f_S(s)$ for Lebesgue-almost every $s>0$.
\end{proof}

\begin{proof}[Proof of Proposition \ref{prop:edf-frailty-katz}]
Conditioning on $M_i$ gives $\mathcal L_{X_i}(t)=P_{M_i}(\phi_i(t))$, and independence yields $\mathcal L_S(t)=\prod_{j=1}^n P_{M_j}(\phi_j(t))$. By independence and identity
$$\E[X_i e^{-tX_i}]=-\frac{d}{dt}P_{M_i}(\phi_i(t))=P_{M_i}'(\phi_i(t))\,\phi_i^{(1)}(t),$$
we obtain $\mathcal L_i^*(t)=\phi_i^{(1)}(t)\,P_{M_i}'(\phi_i(t))\prod_{j\ne i}P_{M_j}(\phi_j(t))$. Applying \eqref{eq:katz-pgf-derivative} with $z=\phi_i(t)$ gives \eqref{eq:Li-star-katz}.
\end{proof}

\begin{proof}[Proof of Lemma \ref{lem:complex-L}]
Fix $\gamma>0$ and let $K\subset\{z\in\mathbb C:\Re z\ge \gamma\}$ be compact. For all $z\in K$, $|e^{-zS}|=e^{-(\Re z)S}\le e^{-\gamma S}\le 1$, so $\mathcal L_S(z)$ is well-defined, and for each $\omega\in\Omega$ the map $z\mapsto e^{-zS(\omega)}$ is entire.

To justify differentiation under $\E[\cdot]$, note that for any integer $m\ge 0$ and $z\in K$,
$$\big|S^m e^{-zS}\big|\le S^m e^{-\gamma S}\le C_{m,\gamma},\qquad C_{m,\gamma}:=\sup_{x\ge 0} x^m e^{-\gamma x}<\infty,$$
where $C_{m,\gamma}=(m/(e\gamma))^m$ for $m\ge 1$ and $C_{0,\gamma}=1$. The bound is deterministic and does not depend on $z\in K$. Dominated convergence applied to the complex difference quotients, therefore, yields that $\mathcal L_S$ is holomorphic on $\{\Re z>0\}$ with
$$\frac{\d^m}{\d z^m}\mathcal L_S(z)=\E\big[(-S)^m e^{-zS}\big],\qquad \Re z\ge \gamma.$$

For $\mathcal L_i^*$, the bound $0\le X_i\le S$ gives $|X_i S^m e^{-zS}|\le S^{m+1} e^{-\gamma S}\le C_{m+1,\gamma}<\infty$ for $z\in K$, so the same argument yields holomorphy and
$$\frac{\d^m}{\d z^m}\mathcal L_i^*(z)=\E\big[(-S)^m X_i e^{-zS}\big],\qquad \Re z\ge \gamma.$$

For part (iii), fix a compact $K'\subset\{(t_1,\dots,t_n):\Re t_j\ge \gamma\ \forall j\}$. For $t\in K'$,
$$\left|X_i \exp \Big(-\sum_{j=1}^n t_j X_j\Big)\right|\le X_i e^{-(\Re t_i)X_i}\le \sup_{x\ge 0} x e^{-\gamma x}=\frac{1}{e\gamma},$$
so the dominated convergence theorem gives
$$\frac{\partial}{\partial t_i}\mathcal L_{\bm X}(t_1,\dots,t_n)=-\E \left[X_i \exp \Big(-\sum_{j=1}^n t_j X_j\Big)\right].$$
Evaluating at $t_1=\cdots=t_n=z$ yields $-\left.\frac{\partial}{\partial t_i}\mathcal L_{\bm X}(t_1,\dots,t_n)\right|_{t_1=\cdots=t_n=z}=\E\big[X_i e^{-zS}\big]=\mathcal L_i^*(z)$.
\end{proof}

\bibliography{ref}
\bibliographystyle{apalike}

\end{document}